\documentclass[11pt]{amsart}
\usepackage{amsmath,amssymb}
\usepackage{color}
\usepackage[all,curve,frame]{xy}

\usepackage[pdftex]{graphicx}

\newtheorem{defi}{Definition}[section]
\newtheorem{sinobservacion}[defi]{}
\newenvironment{sinob}{\begin{sinobservacion} \rm}{\end{sinobservacion} }
\newtheorem{coro}[defi]{Corollary}
\newtheorem{notation}[defi]{Notation}
\newtheorem{lem}[defi]{Lemma}
\newtheorem{rem}[defi]{Remark}
\newtheorem{prop}[defi]{Proposition}
\newtheorem{teo}[defi]{Theorem}
\newtheorem{ej}[defi]{Example}

\newcommand{\A}{\mathcal{A}}

\newcommand{\C}{\mathbf{C}}
\newcommand{\K}{\mathbf{K}}
\newcommand{\D}{\mathbf{D}}
\newcommand{\benu}{\begin{enumerate}}
\newcommand{\enu}{\end{enumerate}}

\begin{document}
\title[The Auslander-Reiten quiver of the category of $m-$periodic complexes]
{The Auslander-Reiten quiver of the category of $m-$periodic complexes}
\author[C. Chaio]{Claudia Chaio}
\address[Claudia Chaio]{Centro Marplatense de Investigaciones Matem\'aticas, Facultad de Ciencias Exactas y
Naturales, Funes 3350, Universidad Nacional de Mar del Plata and CONICET,  Mar del
Plata, 7600, Argentina}
\email{claudia.chaio@gmail.com}

\author[A. Gonz\'alez Chaio]{Alfredo Gonz\'alez Chaio}
\address[Alfredo Gonz\'alez Chaio]{Centro Marplatense de Investigaciones Matem\'aticas, Facultad de Ciencias Exactas y
Naturales, Funes 3350, Universidad Nacional de Mar del Plata,  Mar del
Plata, 7600, Argentina}
\email{agonzalezchaio@gmail.com}

\author[I. Pratti]{Isabel Pratti}
\address[Isabel Pratti]{Centro Marplatense de Investigaciones Matem\'aticas, Facultad de Ciencias Exactas y
Naturales, Funes 3350, Universidad Nacional de Mar del Plata,  Mar del
Plata, 7600, Argentina}
\email{nilpratti@gmail.com}

\author[M. J. Souto Salorio]{Mar\'ia Jos\'e Souto Salorio}
\address[Mar\'ia Jos\'e Souto Salorio]{Campus Industrial de Ferrol, Facultade de Inform\'atica, Universidade da Coru\~{n}a, CP 15071, A Coru\~{n}a, Spain}
{\email{maria.souto.salorio@udc.es}

\date{\today}



\maketitle

$$\text{Dedicated to Raymundo Bautista on his eightieth  birthday}$$

%

\begin{abstract}  Let $\mathcal{A}$  be an additive  $k-$category and $\mathbf{C}_{\equiv m}(\mathcal{A})$ be the category of $m-$periodic objects.  For any integer $m>1$, we study  conditions under which  the compression functor ${\mathcal F}_m: \C^{b}(\mathcal{A}) \rightarrow \mathbf{C}_{\equiv m}(\mathcal{A})$ preserves or reflects irreducible morphisms. Moreover, we  find  sufficient  conditions for the functor   ${\mathcal F}_m $ to be a Galois $G$-covering in the sense of \cite{BL}. If in addition $\A$ is a dualizing category  and $\mbox{mod}\, \A$ has finite global dimension then $\mathbf{C}_{\equiv m}(\mathcal{A})$ has almost split sequences. In particular, for a finite dimensional algebra $A$ with finite strong global dimension we determine how to build the Auslander-Reiten quiver of the category  $\C_{\equiv m}(\mbox{proj}\, A)$. Furthermore,  we study the behavior of sectional paths in $\C_{\equiv m}(\mbox{proj}\, A)$, whenever $A$ is any finite dimensional $k-$algebra over a field $k$.
\end{abstract}


\section*{Introduction}

The category of $m-$periodic complexes, for $m \geq  2,$ has independent interest by itself  but also it is related  to the  orbit categories of the bounded derived category.  The connection  between the orbit categories of the derived categories of an algebra and the $m-$periodic complexes was established  by the compression functor.

In  \cite{PX}, L. Peng and  J. Xiao  showed a relationship  between  these two categories. In particular, the mentioned authors dealt with the $2-$periodic chain complexes over a finite dimensional  hereditary algebra $H$ and proved that the  root category  $\D^b(\mbox{mod}\, H)/[2]$,  studied by D. Happel in \cite{H}, inherits a triangulated structure from $\D^b(\mbox{mod}\, H).$

The orbit category of a triangulated category is not necessarily triangulated itself. However in 2005, B. Keller  devised a triangulated hull for certain orbit categories of the derived categories.  For  any  finite dimensional algebra $A$ of  finite global dimension,  the author defined a triangulated functor from the derived category to the triangulated hull  in such a way that the orbit category $\D^b(\mbox{mod}\, A)/[m]$  embeds into its  triangulated hull. The embedding is an equivalence in case that the orbit category admits a canonical triangle structure, see \cite{Ke}.

Several authors  studied the above problem and found that the category of $m-$periodic complexes is a good framework to deal with.  In particular, the natural functor from the category of bounded complexes to the $m-$periodic category, called the compression functor, plays an important role because when it is dense is precisely the case that the orbit category inherits a triangulated structure from the derived category $\D^b(\mbox{mod} \, A).$

For an additive category ${\mathcal A}$,  the category of $m-$periodic  complexes $\C_{\equiv m} ({\mathcal A})$, together with the relative homotopy category $\K_{\equiv m}({\mathcal A})$ and its derived category $\D_{\equiv m}({\mathcal A})$  in case that ${\mathcal A}$ is abelian, became to have more  and more interest by itself and also related to the compression functor, see between others \cite{F}, \cite{Sa}, \cite{S}, \cite{S-18} and \cite{Z}.

It is known that if ${\mathcal A}$ is  a hereditary additive category then the compression functor is dense. Moreover, the result is also true in case that  the algebra $A$ is derived equivalent to  a hereditary category, see \cite{Ke}, \cite{Sa}, \cite{S-18}. We observe that if this is the case, then  the  essential image of the compression functor is   equivalent to the category of $m-$periodic complexes.

Some  properties  concerning the Auslander–Reiten theory for the category of $m-$periodic complexes  when $H$ is a hereditary algebra were studied in \cite{B}, \cite{Ch},  \cite{CD} and \cite{RZ}.
In case that $H$ is a hereditary algebra, the almost split sequences   of   $\C_{\equiv m}(\mbox{proj}\, H)$  were studied in \cite{CD} considering the almost split sequences in the category of complexes over the hereditary algebra and the fact that the  compression functor is dense in that case.
Recently  in \cite{Ch},  the author got the Auslander-Reiten quiver for $\C_{\equiv m}(\mbox{proj}\, H)$ using  an alternative  approach.

In the general context of linear $k-$categories, the notion of Galois $G-$covering  was defined  in \cite{BL}  and requires the density condition.
The authors showed that a Galois $G$-covering between Krull-Schmidt additive categories preserves  almost split sequences. Inspired by this fact,  a natural question is under which conditions over an additive category $\A$ one can ensure that the compression functor ${\mathcal F}_m: \C^{b}(\mathcal{A}) \rightarrow \mathbf{C}_{\equiv m}(\mathcal{A})$ is  a Galois $G-$covering?

The Auslander-Reiten quiver of a category provides important information of the category by considering only those morphisms and vertices from which one can build all others. The construction of this quiver   is often hard and we need  to manage the irreducible morphisms and indecomposable objects of our category. The existence of Auslander-Reiten sequences which  allow us to construct the quiver is a previous step to have in mind.

For $A$  a  finite-dimensional $k$-algebra of finite global dimension,  it was proved in \cite{F}, Theorem 2.10,  that the triangulated hull of the orbit  category of $\D^b(\mbox{mod}\, A)/[m]$ admits Auslander-Reiten triangles.

On the other hand, the triangulated hull $R$  is  triangle equivalent to $\K_{\equiv m}(\mbox{proj}\, A)$ and also to $\D_{\equiv m}(\mbox{mod}\, A)$ (see Theorem 2.10 in  \cite{Z}). Therefore  $\K_{\equiv m}(\mbox{proj}\, A)$  admits  Auslander-Reiten triangles. We point out that this fact is equivalent to the existence of   almost split sequences in the category $\C_{\equiv m}(\mbox{proj}\, A)$ because the relative homotopy category is the stable category of $\C_{\equiv m}(\mbox{proj}\, A)$, see \cite{R}.

A generalization of  finite dimensional $k-$algebras is the notion of dualizing $k-$categories  introduced by M. Auslander and I. Reiten in \cite{AR}, where they proved that  the category $\mbox{proj}\, A$ is an example of a dualizing category. From \cite{BSZ} we know that  if $\mathcal{A}$ is a dualizing category, then the category $\mathbf{C}_{[1,n]}(\mathcal{A})$ has almost split sequences. Moreover, if $\mathrm{mod}\, \mathcal{A}$  has finite global dimension then $\mathbf{C}^{b}( \mathcal{A})$  has almost split sequences (see \cite{HZ}). Therefore, if we know that the compression functor ${\mathcal F}_m$  is  a Galois $G-$covering, the existence of almost split sequences in $\C^{b}(\mathcal{A}) $  guarantee the existence of almost split sequences in $\mathbf{C}_{\equiv m}(\mathcal{A}).$

In this work, we are interested on the construction of the so called Auslander-Reiten quiver  of the category of $m-$periodic complexes over a finite dimensional algebra with finite strong  global dimension.
The category of complexes of fixed size plays an important role in our arguments and we reserve special attention to the restriction of the compression functor to this category.
In the particular case that $A$ is an iterated tilted algebra,  we  deduce  that  the compression functor  is a Galois $G-$covering in the sense of \cite{BL}. Then we use this fact to obtain the Auslander-Reiten quiver of the category  $\C_{\equiv m}(\mbox{proj}\, A)$.
\vspace{.1in}

The work is organized as follows. We dedicate the first section to recall basic results on $m-$periodic complexes over an additive category ${\mathcal A}$ and also some properties of the compression functor  ${\mathcal F}_m :\C^b ({\mathcal A}) \rightarrow  \C_{\equiv m} ({\mathcal A})$, for any positive integer $m$. Furthermore, we collect  a few results from \cite{BL} concerning Galois $G-$coverings, which are useful to prove some statements in the next section. In particular,  we study properties of the compression functor ${\mathcal F}_m :\C^b (\mbox{proj}\,A) \rightarrow  \C_{\equiv m} (\mbox{proj}\,A)$ when $A$ is a finite dimensional algebra of finite global dimension.
In Section 2, we prove some  facts of the compression functor respect to indecomposable objects, irreducible morphisms and density, whenever ${\mathcal A}$ is an additive category.   Section 3,  is devoted to show conditions over $\A$ that allows us to ensure that  the compression functor is a Galois $G-$covering in the sense of \cite{BL}. In the end of this section,  we study  the dualizing case and prove that under extra hypothesis, the category of $m-$periodic complexes has almost split sequences  and that we can get them from almost split sequences of  $\C^b ({\mathcal A})$. More precisely, we can get them  from almost split sequences of  $\C_{[1,n]} ({\mathcal A})$.  In particular, we get the result for $\A= \mbox{proj}\, A$ in case that $A$ is an iterated tilted algebra. In Section 4, we show how to construct the Auslander-Reiten quiver of  $\C_{\equiv m} (\mbox{proj}\, A)$, for $A$ a finite dimensional algebra with finite strong global dimension. Finally, in the last section, we study the behaviour of sectional paths in $\C_{\equiv m} (\mbox{proj}\, A)$, whenever $A$ is any finite dimensional algebra.
\vspace{.1in}

\thanks {The  first three named authors thankfully acknowledge partial support from EXA/1057/22  from Universidad Nacional de Mar del Plata, Argentina. The fourth author thanks support from Proyecto del Ministerio espa\~nol de Ciencia e Innovaci\'on (MICINN) PID2020-113230RB-C21. The first author is a researcher from CONICET.}

\section{Preliminaries}

Let  $\mathcal{A}$ be an additive $k-$category over a commutative ring $k$
and  ${\mathcal{P}}$  be  the full subcategory consisting of all the  projective objects in $\mathcal{A}$. By Mod-$\A$ we denote the category which consists of all additive functors from $\A^{op}$ to the category of abelian groups. The Yoneda functor allows us  to see $\A$ as a full subcategory of the cocomplete abelian category Mod-$\A.$

We denote by $ \C(\mathcal{A})$ the additive $k-$category which consists of unbounded chain complexes and by $\C^b({\mathcal A})$ the full subcategory of $ \C(\mathcal{A})$ which  consists of the bounded complexes.

Let $A$ be a finite dimensional algebra and $\mbox{mod}\, A$ be the  category of finitely generated $A$-modules. We denote by $\mbox{proj}\, A$ the full subcategory of $\mbox{mod}\, A$ whose objects are the projective $A$-modules.

For the category of finitely generated modules over a finite dimensional algebra of finite global dimension, $\mbox{mod}\, A$, we  consider the Frobenius  category $\C^b({\mbox{proj}\, A})$ and the homotopy category $\K^b({\mbox{proj}\, A})$ which is triangle equivalent to the stable category.

With abuse of notation and by the triangle equivalences  given in \cite{H}, we identify the following triangulated categories
$$\underline{\C^b(\mbox{proj}\, A)}\simeq  \K^b(\mbox{proj}\, A) \simeq \D^b(\mbox{mod}\, A).$$

\begin{sinob}  {\bf The  $m-$periodic complexes.} Let $m$ be a positive integer.

By definition an $m-$periodic complex over $\mathcal{A}$ is a complex $X =(X^{i}, d_{X}^{i})$ such that $X^{i} = X^{j}$ and $d_{X}^{i}=d_{X}^{j}$ for all $i, j \in \mathbb{Z}$ where $i\equiv_{m} j$, that means the integer numbers $i$ and $j$ are congruent module $m.$ The objects $X^{i}$ are in $\mathcal{A}$ and the differentials $d_{X}^{i} : X^i \rightarrow X^{i+1}$ are morphisms in $\mathcal{A}$ such that $d_{X}^{i+1} d_{X}^{i} =0$.

If $X =(X^{i}, d_{X}^{i})$ and $Y =(Y^{i}, d_{Y}^{i})$ are two $m-$periodic complexes, a morphism $f: X \rightarrow Y$  is a sequence of morphisms $f^{i} : X^{i} \rightarrow Y^{i}$  of $\mathcal{A}$ such that $f^{i}=f^{j}$ for all $i, j \in \mathbb{Z}$ with $i\equiv_{m} j$ and such that
$d_i f^{i+1} = f^{i} d_{i+1}$  for all $i \in \mathbb{Z}$. We compose the morphisms from right to left.

We denote by $\mathbf{C}_{\equiv m}(\mathcal{A})$   the category of $m$-periodic complexes over ${\mathcal A}.$
The category $\mathbf{C}_{\equiv m}(\mathcal{A})$  is a Frobenius category and the stable category $\underline{\C_{\equiv m}({\mathcal A})}$ coincides with $\K_{\equiv m}({\mathcal A})$  the relative homotopy category   of  $m-$periodic complexes. If we consider the restriction to the full subcategory ${\mathcal{P}}$ of $\mathcal{A}$,  we get the category  $\C_{\equiv m}({\mathcal{P}})$ and also the category  $\K_{\equiv m}({\mathcal{P}})$  (see \cite{PX}).

For each integer $s$ there is a canonical shift functor $[s]$ on  $\C_{\equiv m}({\mathcal A})$ which maps an $m-$periodic complex $(X,d_X)$  to $(X[s],\delta)$ where $\delta ^i=(-1)^sd_X^{s+i}.$

In the particular case that  $\mathcal{A}$ is the category of finitely generated modules over a finite dimensional algebra of finite global dimension, $\mbox{mod}\, A$, we can consider the  $m-$periodic derived category of $A$, $\D_{\equiv m}(\mbox{mod}\, A),$  that is, the localization of the relative homotopy category with respect to quasi-isomorphisms. Moreover, the categories $\K_{\equiv m}(\mbox{proj}\, A)$ and $\D_{\equiv m}(\mbox{mod}\, A)$ are  triangulated  equivalent and the $m-$periodic derived category is invariant under derived equivalences (see \cite{Z} and \cite{S-18}).
\end{sinob}

With abuse of notation, throughout this paper, for a finite dimensional algebra $A$ with finite global dimension we identify  the following triangulated categories:

$$\underline{\C_{\equiv m}(\mbox{proj}\, A)}\simeq  \K_{\equiv m}(\mbox{proj}\, A)\simeq \D_{\equiv m}(\mbox{mod}\, A).$$

\begin{sinob}{\bf The compression functor.}
By \cite{PX}, we know that  there is an exact functor, named the compression functor, ${\mathcal F}_m: \C^b (\mathcal{A}) \rightarrow  \C_{\equiv m} (\mathcal{A})$ that induces an exact functor between the homotopy categories $\K^b (\mathcal{A})$ and $\K_{\equiv m} (\mathcal{A})$. Moreover, it also induces an exact functor from the homotopy category $\K^b({\mathcal{P}})$ and the relative homotopy category $\K_{\equiv m}({\mathcal{P}}).$

With abuse of notation, we also denote by ${\mathcal F}_m $ the compression functor between the homotopy categories.
\vspace{.05in}

We recall the definition of the compression functor ${\mathcal F}_m: \mathbf{C^{b}}(\mathcal{A}) \rightarrow \mathbf{C}_{\equiv m}(\mathcal{A})$ given in \cite{PX}.

\begin{defi} For $X =(X^{i}, d_{X}^{i}) $ a complex in $\mathbf{C^{b}}(\mathcal{A})$, we set ${\mathcal F}_{m}(X) =({\mathcal F}_{m}(X)^{i}, d_{{\mathcal F}_{m}(X)}^{i})$ where ${\mathcal F}_{m}(X)^{i}= \oplus_{_{t \in \mathbb{Z}}} X ^{i+tm}$ and  $d_{{\mathcal F}_{m}(X)}^{i}= (d_{st}^{i})_{s, t \in \mathbb{Z}}$ such that $d_{st}^{i}: X^{i+sm} \rightarrow X^{(i+1)+tm}$ with $d_{st}^{i}=0$ for $s \neq t$ and $d_{ss}^{i}=d_{X}^{i+ms}$ for all $s\in \mathbb{Z}$.

For a morphism $f=(f^{i}):X \rightarrow Y$ in $\mathbf{C^{b}}(\mathcal{A})$, we set ${\mathcal F}_{m}(f)=(g^{i}): \  {\mathcal F}_{m}(X) \rightarrow {\mathcal F}_{m}(Y)$ where $g^{i}=(g_{st}^{i})$ such that $g_{st}^{i}: X^{i+sm} \rightarrow Y^{i+tm}$  with $g_{st}^{i}=0$ for $s \neq t$ and $g_{ss}^{i}=f^{i+ms}$ for all $s\in \mathbb{Z}$.
\end{defi}

In a more general context,  if $\mathcal{A}$ has coproducts then we can consider  the  compression  functor $\Delta: \C({\mathcal A})\rightarrow \C_{\equiv m}({\mathcal A})$  given by
 $\Delta(X)=\oplus_ {j\in {\mathbb Z}} X[mj ]$ for a complex $(X, d_X)\, \in \C({\mathcal A}).$ When we restrict to bounded complexes we get that
   ${\mathcal F}_m : \C^b ({\mathcal A})\rightarrow  \C ({\mathcal A})\rightarrow \C_{\equiv m}({\mathcal A})$ is the composition of the  inclusion functor     $\C^b({\mathcal A})\rightarrow  \C({\mathcal A})$ with $\Delta$.

We observe that  the inclusion functor ${\bigtriangledown}: \C_{\equiv m}({\mathcal A})\rightarrow  \C({\mathcal A})$ is the right adjoint of $\Delta.$ That means,
$$ \mbox{Hom}_{ \C_{\equiv m}({\mathcal A})}(\Delta(X), \Delta (Y))
 \simeq    \mbox{Hom}_ {\C ({\mathcal A})} (X, {\bigtriangledown} \Delta (Y)).$$
\vspace{.05in}

Now,  we focus on the case that $\A=\mbox{proj}\, A$  for   $A$  a finite dimensional algebra of  finite global dimension.

The following result shows the link between the compression functor and the orbit categories, see \cite{Ke}.

\begin{teo}\label{tritodas}  Let $A$ be a finite dimensional algebra with finite global dimension  and $m\geq 1$. The compression functor ${\mathcal F}_m: \  \K^b (\emph{proj}\, A)\rightarrow    \K_{\equiv m}(\emph{proj}\, A)$ yields   an embedding  from its essential image $\K^b (\emph{proj}\, A)/[m]$ into its triangulated hull $\K_{\equiv m}(\emph{proj}\, A).$

Moreover, the following statements are equivalent.
\begin{enumerate}
\item[(a)] The category $ \K^b (\emph{proj}\, A)/[m]$  has a canonical  triangulated structure.
\item[(b)] The embedding $i: \K^b (\emph{proj}\, A)/[m]{\longrightarrow }\K_{\equiv m}(\emph{proj}\, A)$ is an equivalence.
\item[(c)] The compression functor ${\mathcal F}_m: \  \K^b (\emph{proj}\, A)\rightarrow    \K_{\equiv m}(\emph{proj}\, A)$ is dense.
\end{enumerate}
Furthermore, if one of the above Statements holds for one choice of $m$ then it holds for all positive integer $m.$
\end{teo}
\begin{proof} By \cite[Theorem 4.3]{S-18} and also \cite{Z}, we know that  the compression functor provides   an embedding $i$  from its essential image $\K^b (\mbox{proj}\, A)/[m]$ to its triangulated hull $\K_{\equiv m}(\mbox{proj}\, A)$.
 Moreover, by  \cite{Z}, we know that the triangulated hull of $\K^b (\mbox{proj}\, A)/[m]$  is triangle equivalent to the category $\K_{\equiv m}(\mbox{proj}\, A).$  Then we get that Statement (a) is equivalent to Statement (b).

On the other hand, we  have  that the embedding $i: \K^b (\mbox{proj}\, A)/[m]{\longrightarrow }\K_{\equiv m}(\mbox{proj}\, A)$ is an equivalence   if and only if $i$   is dense. Moreover,  $i$   is dense  if and only if
 the compression functor ${\mathcal F}_m$ is dense. Then  Statement (b) is equivalent to Statement (c).

By  \cite[Theorem 1]{S},  we know that if the orbit category $\K^b (\mbox{proj}\, A)/[m]$ is triangulated for one choice of $m$ then it is triangulated for each $m$, proving the result.
\end{proof}
\end{sinob}

\begin{sinob}{\bf Galois $G$-covering.} \label{pre-cov}
We recall  several notions and results stated in \cite{BL} that shall be fundamental to prove some results throughout the next sections. Here  a linear category stands for an additive skeletal $k-$linear category.
\vspace{.05in}

We recall  the definition of $G$-precovering given in \cite{BL}, (also see \cite[1.6]{A}).
\vspace{.05in}

\begin{defi} \label{pre-covering} Let $\mathcal{A}$ and $\mathcal{B}$ be linear categories with $G$ a group acting on $\mathcal{A}$. A functor
$F : \mathcal{A} \rightarrow  \mathcal{B}$ is called a $G$-precovering provided that $F$ has a $G$-stabilizer $\delta$ such that, for
any $X, Y \in \mathcal{A}_{0}$, the following two maps are isomorphisms:

$$F_{X,Y} :\oplus \mathcal{A}(X, g . Y )\rightarrow \mathcal{B}(F(X), F(Y)): (u_g)_{g \in G} \rightarrow \Sigma_{g \in G}F(u_g)\delta_{(g, Y)}.$$

$$F^{X,Y} :\oplus \mathcal{A}(g .X, Y )\rightarrow  \mathcal{B}(F(X), F(Y)): (v_g)_{g \in G} \rightarrow \Sigma_{g \in G}\delta^{-1}_{(g, X)} F(v_g).$$
\end{defi}

In the above definition, as observed by H. Asashiba in \cite[1.6]{A}, it is sufficient to require that all  $F_{X,Y}$ or all $F^{X,Y}$ be isomorphisms.
\vspace{.05in}

Next, we state the definition of an irreducible morphism in an additive category.

\begin{defi} \label{irred} Let $\mathcal{A}$ be an additive category
and $f:X\rightarrow Y$ a morphism in $\mathcal{A}$. The morphism
$f$ is said to be a {\em section} (respectively, {\em a
retraction}) if and only if there is a morphism $h:Y \rightarrow
X$ such that $hf = 1_X$ (respectively,  $fh = 1_Y$). Whenever one of
these conditions hold,  $f$ is said to be a {\em split} morphism.

The morphism $f$  is said to be {\em irreducible} if
it is not a split morphism and, for any factorization $f = hg$, we
have that either $g$ is a section or $h$ is a retraction.
\end{defi}

The next result is fundamental to prove that the compression functor ${\mathcal F}_m$ reflects irreducible morphisms.  Although the isomorphisms that we use to prove the result are well-known and moreover, since we can find a proof in Proposition 3.9 (2), in  the pre-print  \cite{Sa}, for the convenience of the reader, bellow we state and present a proof that the functor ${\mathcal F}_m$ is a $G-$precovering with $G$ the cyclic group generated by  $[m]$.
\vspace{.1in}

 First observe that   the cyclic group generated by $[m],$   acts admissibly on $ \C^b ({\mathcal A}).$  In fact,
\begin{enumerate}
\item the action is free because $X\neq X[mj]$ for any indecomposable complex   in $ \C^b ({\mathcal A})$ and
 \item it  is locally bounded since  $\mbox{Hom}(X,Y[mj])=0$ for all but finitely many integers $j$ and for any indecomposable complexes $X$ and $Y$ in $\C^b ({\mathcal A}).$
 \end{enumerate}

 \begin{prop} \label{precov} Let ${\mathcal A}$ be an additive $k-$category. The functor ${\mathcal F}_m : \C^b ({\mathcal A}) \rightarrow \C_{\equiv m}({\mathcal A})$ is a $G-$precovering  with $G$ the cyclic group generated by $[m]$ and  $\delta =id$ a $G-$stabilizer.
\end{prop}
 \begin{proof}
Consider $X$ and $Y$ in $\C^b({\mathcal A}).$  Since  the complexes are bounded then  there is a finite set $I \subset \mathbb{Z}$ such that $\mbox{Hom}_ {\C^b ({\mathcal A})} (X, Y[mj])=0$ for all integer $j\notin I$.

Using this fact we have the following  isomorphisms:

 $$
 \begin{array}{lcl}
  \mbox{Hom}_{ \C_{\equiv m}({\mathcal A})}({\mathcal F}_m (X), {\mathcal F}_m (Y))
 &=&  \mbox{Hom}_ {\C ({\mathcal A})} (X,   {\mathcal F}_m (Y))\\
 &=&   \mbox{Hom}_ {\C ({\mathcal A})} (X, \oplus_ {j\in \mathbb{Z} } Y[mj] ) \\
  &=&  \oplus_ {j\in I}  \mbox{Hom}_ {\C^b ({\mathcal A})} (X, Y[mj])\\
  &=&  \oplus_ {j\in  \mathbb{Z} }  \mbox{Hom}_ {\C^b ({\mathcal A})} (X, Y[mj])
\end{array}$$

\noindent proving that  the functor ${\mathcal F}_m$ is a $G-$precovering, (see  Definition \ref{pre-covering}).
\end{proof}

As a consequence of Proposition \ref{precov}, we get the following corollary  which is fundamental to prove the irreducibility of the morphisms in different categories.

We first observe that in \cite[Lemma 2.7]{BL} the complexes  $X, Y$ and $Z$ in $\C^b ({\mathcal A})$ are not necessarily indecomposable.

\begin{coro}\emph{(\cite{BL}, Lemma 2.7)} \label{section}   Let ${\mathcal A}$ be an additive category and   $f:X\rightarrow Y$ in $\C^b ({\mathcal A})$ be  a morphism.
The following conditions hold.
\begin{enumerate}
\item  If $g:X\rightarrow Z$ or $g:Z\rightarrow Y$  is a morphism in $\C^b ({\mathcal A}) $  then $g$ factorizes through $f$ if and only if
${\mathcal F}_m(g) $ factorizes through ${\mathcal F}_m(f).$
\item The morphism $f$ is a section, retraction, or isomorphism if an only if ${\mathcal F}_m(f) $ is a section, retraction, or isomorphism, respectively.
\end{enumerate}
\end{coro}

A functor $F : \mathcal{A}\rightarrow \mathcal{B}$ between linear categories is called almost dense if each indecomposable object in $\mathcal{B}$ is isomorphic to an object lying in the image of $F$.
\vspace{.05in}

Finally, we recall the definition of a Galois $G$-covering stated in \cite[Definition 2.8]{BL}.

\begin{defi}\label{Gal} Consider $ \mathcal{A}$ and $\mathcal{B}$ linear categories with $G$ a group acting admissibly on $\mathcal{A}$. Following \cite[Definition 2.8]{BL} we say that a $G$-precovering $F : \mathcal{A} \rightarrow \mathcal{B}$ is called a Galois $G$-covering provided that the following
conditions are verified.
\begin{enumerate}
\item[(a)] The functor $F$ is almost dense.
\item[(b)] If $X\in  \mathcal{A}$ is indecomposable, then $F(X)$ is indecomposable.
\item[(c)] If $X, Y\in  \mathcal{A}$ are indecomposable with $F(X) = F(Y)$, then there exists some $g \in  G$
such that $Y = g .X$.
\end{enumerate}
\end{defi}
\end{sinob}

\begin{sinob}{\bf The category of complexes of fixed size.} \label{notation}  Let ${\mathcal A}$ be an additive category. The category of complexes of fixed sized, precisely of size $n$ with $n \geq 2$,  is the full subcategory of  $\C^b(\mathcal{A})$    whose objects are the complexes with cells zero
outside of the natural interval $[1,n].$ We denote by  $\C_{[1,n]}(\mathcal{A})$ this category.

If $0<n\leq s$ we can consider  the full and faithful functor  $\phi_{n,s}: \,  \C_{[1,n]} ({\mathcal A})\rightarrow  \C_{[1,s]} ({\mathcal A})$    that sends each $X \in \C_{[1,n]}({\mathcal A}) $  to $$  X^{1} {\rightarrow} X^{2} \rightarrow \dots    {\rightarrow} X^{n}\stackrel{}{\rightarrow }0 \stackrel{s-n}{\cdots}{\rightarrow }0$$  in  $\C_{[1,s]}({\mathcal A}).$ Therefore we can consider $\C_{[1,n]} ({\mathcal A})$ as a full subcategory of $\C_{[1,s]} ({\mathcal A}).$

Without loss of generality, if $n\leq s$ we shall identify  a complex $X\in \C_{[1,n]}({\mathcal A}) $ with $\phi_{n,s}(X) \in \C_{[1,s]}({\mathcal A}).$
\end{sinob}

\section{The compression functor of bounded complexes}

Throughout all this section, we consider $\mathcal{A}$ an additive category.

\subsection{The compression functor and the indecomposable complexes}
We start  this section showing that  the compression functor provides an equivalence between the category $\C_{[1,n]}({\mathcal A})$ and a full subcategory of $\C_{\equiv s}({\mathcal A}),$ in the case that $0<n\leq s$.
\vspace{.05in}

First, we fix some useful notation.

 \begin{notation} \label{X_m}
\emph{We denote by $ \mathcal{X}_{s}'$ the full subcategory of $\C_{\equiv s}({\mathcal A})$ which consists of the $s-$periodic complexes  $Z'$ as follow}

$$Z': \cdots \rightarrow Z^{s-1}\stackrel{}{\rightarrow }Z^s \stackrel{0}{\rightarrow } Z^{1} {\rightarrow} Z^{2} \rightarrow \dots    {\rightarrow} Z^{s} \stackrel{0}{\rightarrow }Z^{1} {\rightarrow} Z^{2} \rightarrow \dots    {\rightarrow} Z^{s}   \stackrel{0}{\rightarrow }Z^{1} {\rightarrow}   \dots $$

\noindent \emph{with differential $d^{sj}=0$ for any $j\in {\mathbb Z}$.}
\vspace{.1in}

\emph{Whenever it is convenience,  we use  $\mathcal{X}_{n, s}' $ to denote  the full subcategory of $\mathcal{X}_{s} '$ which consists of the complexes  $Z' \in  \mathcal{X}_{s}'$ such that
$Z^{t+sj}=0$ for any $t\in \{n+1, \dots , s\}$ and  $j\in {\mathbb Z}.$
 Note that $ \mathcal{X}_{s, s} ' = \mathcal{X}_{s} ' .$}
\vspace{.05in}

\emph{By  $(-)' : \,  \C_{[1,s]}({\mathcal A})\rightarrow   \mathcal{X}_{s} '  $ we denote  the functor that sends  $Z\in \C_{[1,s]}({\mathcal A})$  to  $Z' \in  \mathcal{X}_{s} '.$}
\end{notation}

 \begin{rem} \label{ff}
\emph{ The  functor $(-)':\,  \C_{[1,s]}({\mathcal A})\rightarrow   \mathcal{X}_{s}'$  is an equivalence because it is  essentially surjective and also a full and faithful functor. In fact, consider $X$ and $Y$  complexes  in $ \C_{[1,s]}  ({\mathcal A})$. Observe that   $f:X\rightarrow Y$ is a morphism of complexes if and only if  $f':X'\rightarrow Y'$  is a morphism in $\C_{\equiv s}({\mathcal A}).$   This means,
 $\mbox{Hom}_{ \C_{\equiv s}({\mathcal A})}(X', Y')= \mbox{Hom}_ {\C_{[1,n]}({\mathcal A})}(X, Y)$  and therefore the functor  $(-)'$
 is full and faithful.}
 \end{rem}

 By Remark \ref{ff}, we get the following result.

\begin{prop}   \label{m>n}  Consider  $s\geq n.$ The compression functor ${\mathcal F}_s : \C_{[1,n]} ({\mathcal A}) \rightarrow \C_{\equiv s}({\mathcal A})$ is  a full and faithful functor and   its essential image is the full subcategory  $ \mathcal{X}_{n,s}' $ of $\C_{\equiv s}({\mathcal A}).$ Then we get the factorization ${\mathcal F}_s : \C_{[1,n]} ({\mathcal A}) \stackrel{\simeq}\longrightarrow  \mathcal{X}_{n,s}'  \stackrel{i}\rightarrow  \C_{\equiv s}({\mathcal A}).$ Moreover,  ${\mathcal F}_s$ preserves and reflects indecomposable complexes.
\end{prop}

\begin{proof}
Consider  $Z \in \C_{[1,n]} ({\mathcal A})$   and $Z' \in  \mathcal{X}_{n,s}'$ the $s-$periodic complex defined in \ref{X_m}.

We affirm  that  ${\mathcal F}_s (Z) =Z'.$ In fact,    for each ${t \in {\mathbb Z}}$  we can consider    ${j_0\in {\mathbb Z}}$   the   unique integer such that $t\in [sj_0+1, sj_0+s].$ Then we have that  ${\mathcal F}_s (Z)^t= \oplus_ {j\in {\mathbb Z}} Z^{t-sj}=
Z^{t-sj_0}=(Z')^t .$ We infer that  the essential image of  ${\mathcal F}_s $ is the full subcategory  $ \mathcal{X}_{n,s}' $ into $\C_{\equiv s}({\mathcal A}).$

By Remark \ref{ff}, we  get that  $(-)' \circ \phi_{n,s}:  \C_{[1,n]}({\mathcal A})\rightarrow   \mathcal{X}_{n,s}'$ is an equivalence of categories.
 Then the functor ${\mathcal F}_s$ is the composition
$\C_{[1,n]} ({\mathcal A}) \stackrel{(-)' \circ \phi_{n,s}}\longrightarrow  \mathcal{X}_{n,s}'  \stackrel{i}\rightarrow  \C_{\equiv s}({\mathcal A})$
\noindent where $i$ is the inclusion morphism.

Finally, consider  $X$  a complex  in $ \C_{[1,n]}  ({\mathcal A})$. Note that  $W$ is  a direct summand of
 ${\mathcal F}_s (X)=X' $  in $ \C_{\equiv s}({\mathcal A})$ if and only if  $W\in  \mathcal{X}_n '.$  Indeed,  if $s>n$  we conclude the result using the fact that  $(X')^j =0$  for each $j\in \{n+1, \dots , s\}.$   If $s=n$  since there is a section $u:W\rightarrow X'$ then $d_W^n=0$ because $u^1d_W^n=d_{X'}^n u^n=0$ and   $u^1$ is a monomorphism.
 Therefore,  $W\in  \mathcal{X}_n '$  and $W=Y'={\mathcal F}_s (Y)$    for some $Y$   in $ \C_{[1,n]}  ({\mathcal A})$.
From this fact and using that  the compression functor is additive we get the following equivalent statements: ${\mathcal F}_s (X)={\mathcal F}_s (Y)\oplus {\mathcal F}_s (L)$ in $\C_{\equiv s}({\mathcal A})$  if and only if ${\mathcal F}_s (X)={\mathcal F}_s (Y)\oplus {\mathcal F}_s (L)$ in $ \mathcal{X}_n '$ and also   ${\mathcal F}_s (X)={\mathcal F}_s (Y)\oplus {\mathcal F}_s (L)$ in $ \mathcal{X}_n '$ if and only if  $X=Y\oplus L$  in $ \C_{[1,n]}  ({\mathcal A}).$
\end{proof}

We observe that, up to shifts, a complex in  $\C^{b}({\mathcal A})$ belongs to a category $\C_{[1,n]} ({\mathcal A})$, for some $n$.

\begin{rem} \label{obs}
\emph{Consider $l(X)$ the size of the complex  $X \in \C^{b}({\mathcal A}).$
Then the following conditions hold.}
\emph{\begin{enumerate}
\item  For each  complex $X \in \C^b  ({\mathcal A})$ there is an integer $t$   such that  $X[t]\in  \C_{[1,n]} ({\mathcal A})$ where $n=l(X)$. Moreover, ${\mathcal F}_m(X[t]) \in  \mathcal{X}_n'$  for every $m \geq n=l(X)$, see Proposition \ref{m>n}.
\item For every $X$ and $Y$ in $\C^b  ({\mathcal A})$ there is an integer  $t$  and $n>1$ such that $X[t]$ and $Y[t]$ in  $\C_{[1,n]} ({\mathcal A})$. Moreover,  the following isomorphisms $$\mbox{Hom}_{ \C_{\equiv m}({\mathcal A})}({\mathcal F}_m (X[t]), {\mathcal F}_m (Y[t]))= \mbox{Hom}_ {\C_{[1,n]}  ({\mathcal A})} (X[t], Y[t])=  \mbox{Hom}_ {\C^b ({\mathcal A})} (X, Y)$$
\end{enumerate}}
\noindent \emph{hold as a consequence of the equivalence given in Proposition \ref{m>n} and the fact that  the shift functor  $[-]: \C^b  ({\mathcal A})\rightarrow \C^b  ({\mathcal A})$ is an equivalence of categories.}
\end{rem}

In general, for bounded complexes we shall show that compression functor preserve indecomposable complexes.
\vspace{.05in}

First, we introduce the following notation.

\begin{notation}\label{X'} \emph{We fix a $m>0$ and a $n>1.$}

\emph{We denote  by  $w$  the positive integer such that $ w=\mbox{min} \{s \geq 1, \, \mbox{such that }   ms \geq n \}.$}

\emph{We note that there is an integer $l\geq 0 $ such that $mw=n+l$.}
\end{notation}

Now, we  collect  some facts about   the compression functor.

 \begin{lem} \label{descompos}   For $m \geq 1$ and $X\in \C_{[1,n]} ({\mathcal A})$  the following conditions hold.
\begin{enumerate}
\item ${\mathcal F}_{mw}(X) =\oplus_ {j\equiv_w 0} X[mj ] =X'  \in \mathcal{X}_{n,mw} '.$
\item The   $m-$periodic complex ${\mathcal F}_m(X) $  decomposes as  a finite sum of $mw-$periodic complexes as follows
$${\mathcal F}_m(X) =  \oplus_ {j\in [0,w-1] } {\mathcal F}_{mw}(X)[mj].$$
\end{enumerate}
\end{lem}
\begin{proof}
 If  $s \geq n,$  then by Proposiion \ref{m>n} we know that ${\mathcal F}_s(X) =\oplus_ {j\in {\mathbb Z}} X[sj ] =X' \in   \mathcal{X}_{s} '$. In particular, $mw \geq n$ and we have that

 $${\mathcal F}_{mw}(X) =  \oplus_{j\in {\mathbb Z}} X[mwj]= \oplus_ {j\equiv_w 0} X[mj] =  X' \in  \mathcal{X}'_{n,mw}.$$

For $m\geq 1$, we can express
$$ \begin{array}{lcl}
   {\mathcal F}_m(X) & =& \oplus_ {j\in {\mathbb Z}} X[mj ]\\&=& \oplus_ {j\equiv_w 0} X[mj] \oplus \oplus_ {j\equiv_w 1} X[mj]   \oplus  \dots \oplus_ {j\equiv_w (w-1)} X[mj ] \\
   &=& \oplus_ {j\equiv_w 0} X[mj] \oplus \oplus_ {j\equiv_w 0} X[mj+1]   \oplus  \dots \oplus_ {j\equiv_w 0} X[mj +(w-1)].
   \end{array}
   $$
Therefore, we  get  that

$$ {\mathcal F}_m(X) =  \oplus_ {j\in [0,w-1] } X'[mj] = \oplus_ {j\in [0,w-1] } {\mathcal F}_{mw}(X)[mj],$$

\noindent proving the result.
\end{proof}

Now, we prove that the functor
 ${\mathcal F}_m: \C^{b}(\mathcal{A}) \rightarrow \mathbf{C}_{\equiv m}(\mathcal{A})$ preserves indecomposable complexes.

\begin{prop}\label{ind} Let $\mathcal{A}$  be a Krull-Schmidt category.  For any $m>1$ we get that the compression functor ${\mathcal F}_m: \C^{b}(\mathcal{A}) \rightarrow \mathbf{C}_{\equiv m}(\mathcal{A})$ preserves and reflects indecomposable complexes.
\end{prop}

\begin{proof}  If $Z$ is a direct summand of $X$ in $ \C^{b}(\mathcal{A})$,  that is, $X=Z\oplus L$  in $\C^b({\mathcal A})$ for some $L$  then
${\mathcal F}_m (X)={\mathcal F}_m (Z)\oplus {\mathcal F}_m (L)$ in $\C_{\equiv m}({\mathcal A}),$ since the compression functor is additive.

Consider an indecomposable complex $X\in \C^b ({\mathcal A})$. Without loss of generality, we may assume that $X\in \C_{[1,n]} ({\mathcal A})$ with $X^1\neq 0$ and $X^n\neq 0.$ Since $X$ is indecomposable  in $ \C^{b}(\mathcal{A})$, then so is in $ \C_{[1,n]} ({\mathcal A}).$ If $m\geq  n,$  we get the result from Proposition  \ref{m>n}.

Now, we analyze the case that $m<n.$ By Lemma \ref{descompos}, the complex

$${\mathcal F}_m(X) =\oplus_ {j\in [0,w-1] } {\mathcal F}_{mw}(X)[mj]$$

\noindent is  a finite sum of $mw-$periodic complexes with $mw=n+l$ where $0 \leq l \leq m-1.$ In fact, since $w=\mbox{min} \{s \geq 1, \, \mbox{such that}  \, ms \geq n \}$ then $mw \geq n$ and by the minimality of $w$ we have that $n> m(w-1)$. Therefore, $n+l \geq n$ implies that $l \geq 0$. Furthermore,
$n+l-m < n$ then $l < m$ and hence $l \leq m-1$. We observe that in this case $w\geq 2.$

By Proposition  \ref{m>n} we know that ${\mathcal F}_{mw}(X)$ is indecomposable in $ \C_{\equiv mw}({\mathcal A})$   because $X$ is indecomposable  in  $\C_{[1,n]} ({\mathcal A}).$ Then also the complex ${\mathcal F}_{mw}(X)[mj]$  is indecomposable in $\C_{\equiv mw}({\mathcal A})$ for any ${j\in [0,w-1] }.$

 We  prove that  ${\mathcal F}_m(X)$ has only trivial  direct summand    in  $ \C_{\equiv m}({\mathcal A}).$ In fact, if $Z$ is a non-zero direct summand of  ${\mathcal F}_m(X)$   in  $ \C_{\equiv m}({\mathcal A})$ then   there is a section  $u:Z  \rightarrow {\mathcal F}_m(X)$  in  $ \C_{\equiv m}({\mathcal A})$. Furthermore, $u$ is also a section  in  $\C({\mathcal A})$ and also a section in  $ \C_{\equiv mw}({\mathcal A}).$
Then we infer that  there is    ${J \subseteq [0,w-1] }$ such that  $Z=\oplus_ {j\in J} {\mathcal F}_{mw}(X)[mj] $ as $mw-$periodic complex.
Since $Z \neq 0$ then  $J\neq \emptyset$. Consider  $j_0 \in J.$

 Since $Z$ is  a $m-$periodic complex we have that $Z=Z[m(t-j_0)],$  for every  ${t\in [0,w-1] }.$  Then

 $$ \begin{array}{lcl}
   Z & =& \oplus_ {j\in J} {\mathcal F}_{mw}(X)[m(j+t-j_0)]\\
   &=&{\mathcal F}_{mw}(X)[mt] \oplus \oplus_ {j\neq j_0} {\mathcal F}_{mw}(X)[m(j+t-j_0)],
\end{array}
   $$

\noindent and then ${\mathcal F}_{mw}(X)[mj]$ is a direct summand of $Z$  for each ${t\in [0,w-1] }$. Thus $Z={\mathcal F}_m(X)$ proving that   ${\mathcal F}_m(X)$ is an indecomposable $m-$periodic complex.
\end{proof}

\subsection{\bf Irreducible morphisms under the compression functor.}
In the rest of this section,  we focus our attention on the behaviour of the irreducible morphisms under the compression functor ${\mathcal F}_m : \C^b ({\mathcal A}) \rightarrow \C_{\equiv m}({\mathcal A})$.
\vspace{.1in}

First, we   prove that the compression functor reflects irreducible morphisms.

\begin{lem}\label{bycoro} Let $f:X \rightarrow Y$  be a morphism in $\mathbf{C}^{b}(\mathcal{A})$. If ${\mathcal F}_m(f): {\mathcal F}_m(X) \rightarrow {\mathcal F}_m(Y)$ is irreducible in $\mathbf{C}_{\equiv m}(\mathcal{A})$ then $f$ is irreducible in $\mathbf{C}^{b}(\mathcal{A})$.
 \end{lem}
\begin{proof}  The morphism $f$ is neither a section nor a retraction, otherwise by Corollary  \ref{section} (2),   ${\mathcal F}_m(f)$ is a section or a retraction, getting a contradiction to the fact that ${\mathcal F}_m(f)$ is irreducible.

On the other hand, if    $f=gh$    in $\C^b ({\mathcal A})$ then  ${\mathcal F}_m(h)$ is a section or ${\mathcal F}_m(g)$  is a retraction. Again by Corollary \ref{section} (2), $h$ is a section or $g$ is a retraction, getting the result.
\end{proof}

Now, we determine conditions under which  the compression functor ${\mathcal F}_m: \C^{b}({\mathcal A}) \rightarrow \mathbf{C}_{\equiv m}(\mathcal{A})$ preserves irreducible morphisms.

\begin{teo} \label{irred} Let $f:X \rightarrow Y$  be a morphism in $\mathbf{C}^{b}(\mathcal{A})$.  Let $ \ell(X\oplus Y)$  be the size of the direct sum of the complexes $X$ and $Y$. For $m \geq   \ell(X\oplus Y)$ we have that ${\mathcal F}_m(f)$ is irreducible in  $\mathbf{C}_{\equiv m}(\mathcal{A})$    if and only if $f$ is irreducible in $\C^{b}(\mathcal{A})$.
\end{teo}
\begin{proof} Consider   $m\geq   \ell(X\oplus Y)$. The implication that if ${\mathcal F}_m(f)$ is irreducible in  $\mathbf{C}_{\equiv m}(\mathcal{A})$  then
$f$ is irreducible in $\C^{b}(\mathcal{A})$ holds by Lemma \ref{bycoro}.

Consider   $f:X \rightarrow Y$  an  irreducible in $\C^{b}(\mathcal{A})$. First, we prove the case that $f:X \rightarrow Y$ is a morphism in $\mathbf{C}_{[1,m]}(\mathcal{A})$  and $m= \ell(X\oplus Y).$

If $f:X \rightarrow Y$ is irreducible in $\C^{b}(\mathcal{A})$ then it is not a section neither a retraction.  Then by Corollary \ref{section} (2), we have that ${\mathcal F}_m(f):{\mathcal F}_m (X) \rightarrow {\mathcal F}_m(Y)$ neither so is.

Observe that that both  $  \ell(X)\leq   \ell(X\oplus Y) \leq m$ and $  \ell(Y)\leq   \ell(X\oplus Y) \leq m.$  By Proposition \ref{m>n} we get
${\mathcal F}_m (X)=X'  \in \mathcal{X}_{m} '$ and ${\mathcal F}_m (Y) =Y'  \in \mathcal{X}_{m}'.$

Assume that ${\mathcal F}_m(f)$ factors through a complex $Z$ in $\mathbf{C}_{\equiv m}(\mathcal{A})$. Then we have the following commutative diagram

\begin{displaymath}
\xymatrix {X'  \ar[d]^{g} &  X ^{1} \ar[r]^{d_{X}^{1}} \ar[d]^{g^{1}} & X ^{2}\ar[r]^{d_{X}^{2}} \ar[d]^{g^{2}}&  \cdots \ar[r] &  X ^{m} \ar[r]^{0}\ar[d]^{g^{m}}& X ^{1}\ar[d]^{g^{1}}  \\
Z \ar[d]^{h}& Z^{1} \ar[r]^{d_{Z}^{1}} \ar[d]^{h^{1}} & Z^2 \ar[r]^{d_{Z}^{2}}\ar[d]^{h^{2}}& \cdots \ar[r] & Z^{m} \ar[r]^{d_{Z}^{m}}\ar[d]^{h^{m}} & Z^{1}\ar[d]^{h^{1}}  \\
Y'& Y^{1} \ar[r]^{d_{Y}^{1}}  & Y^2 \ar[r]^{d_{Y}^{2}}& \cdots \ar[r] & Y^{m} \ar[r] & Y^{1} \\}
 \end{displaymath}

\noindent that induces the following commutative diagram in $\C^{b}(\mathcal{A})$.

\begin{displaymath}
\xymatrix {X  \ar[d]^{\widetilde{g}} & 0 \ar[r] \ar[d]&  X ^{1} \ar[r]^{d_{X}^{1}} \ar[d]^{g^{1}} & X ^{2}\ar[r]^{d_{X}^{2}} \ar[d]^{g^{2}}&  \cdots \ar[r] &  X ^{m} \ar[r]^{0}\ar[d]^{g^{m}}& 0 \ar[d]  \\
Z \ar[d]^{\widetilde{h}}& Z^{m} \ar[r]^{d_{Z}^{m}}\ar[d]^{0}& Z^{1} \ar[r]^{d_{Z}^{1}} \ar[d]^{h^{1}} & Z^2 \ar[r]^{d_{Z}^{2}}\ar[d]^{h^{2}}& \cdots \ar[r] & Z^{m} \ar[r]^{d_{Z}^{m}}\ar[d]^{h^{m}} & Z^{1}\ar[d]^{h^{1}}  \\
Y& 0 \ar[r]& Y^{1} \ar[r]^{d_{Y}^{1}}  & Y^2 \ar[r]^{d_{Y}^{2}}& \cdots \ar[r] & Y^{m} \ar[r]^{0} & 0 \\}
 \end{displaymath}

By hypothesis $f$ is irreducible in
$\C^{b}(\mathcal{A})$. Hence $\widetilde{g}$ is a section or $\widetilde{h}$ is a retraction. In case that $\widetilde{g}$ is a section  then by Corollary \ref{section} (2) we get that $g$ is a section. Similarly, if $\widetilde{h}$ is a retraction then again, by Corollary \ref{section} (2),  $h$ so is. Hence ${\mathcal F}_m(f)$  is irreducible, as we wish to prove.

Now, we analyze the case where $f:X \rightarrow Y$ is a morphism in $\mathbf{C}_{[1,m]}(\mathcal{A})$  and $m>  \ell(X\oplus Y).$ Assume that $\ell(X\oplus Y)=r.$

Knowing that $\C_{[1,r]} ({\mathcal A})$ is a full subcategory of $\C_{[1,m]} ({\mathcal A}) \subset \C^{b}(\mathcal{A}),$ then every morphism $f:X \rightarrow Y$   in $\mathbf{C}_{[1,r]}(\mathcal{A})$
 can  be consider  as a  morphism in $\C_{[1,m]} ({\mathcal A})$. Then, by the above argument we get that  ${\mathcal F}_m(f): {\mathcal F}_m(X) \rightarrow{\mathcal F}_m(Y)$ is irreducible in $\mathbf{C}_{\equiv m}(\mathcal{A})$ if and only if  $f:X \rightarrow Y$ is irreducible in $\C^{b}(\mathcal{A})$.

Finally, we analyze the general case, that is, when   $f:X \rightarrow Y$  is an irreducible morphism in $\mathbf{C}^{b}(\mathcal{A}).$

We know that there is an interval $[a,b]$ such that $X\oplus Y\in \mathbf{C}_{[a,b]}(\mathcal{A}).$
Then  $(X\oplus Y)[a]\in \mathbf{C}_{[1,n]}(\mathcal{A})$  and $f[a]$ is a morphism in $\mathbf{C}_{[1,n]}(\mathcal{A})$ with $ n=b-a=  \ell(X\oplus Y).$

We observe that $f$ is an irreducible morphism in $\mathbf{C}^{b}(\mathcal{A})$ if and only if  $f[a]$ so is. Moreover,  using that $\mathbf{C}_{\equiv m}(\mathcal{A})$ is closed under shifts, we have that  ${\mathcal F}_m(f[a])$ is irreducible in $\mathbf{C}_{\equiv m}(\mathcal{A})$  if and only if  ${\mathcal F}_m(f)$  is also irreducible. Therefore, if $m \geq  \ell(X\oplus Y)$ then  we conclude the result using the arguments done in the previous cases.
\end{proof}

\begin{rem}\label{0n+1} \emph{ Let  $f:X \rightarrow Y$ be a morphism in
  $\mathbf{C}_{[1,n]}(\mathcal{A}).$  We observe that $f$  is irreducible  in $\mathbf{C}^{b}(\mathcal{A})$  if and only if $f$ is irreducible in   $\mathbf{C}_{[0,n+1]}(\mathcal{A}).$  In fact,  if $f$ is irreducible in $\mathbf{C}^{b}(\mathcal{A})$ then it is irreducible  in any  interval.}

\emph{ Now, if $f$ is irreducible in   $\mathbf{C}_{[0,n+1]}(\mathcal{A})$ in order to prove that $f$ is irreducible in $\mathbf{C}^{b}(\mathcal{A})$
  it is enough to focus on the interval $[0,n+1]$, as we can  see from  the following diagram }
\begin{displaymath}
\xymatrix {X  \ar[d]^{\widetilde{g}} & 0 \ar[r] \ar[d]&  X ^{1} \ar[r]^{d_{X}^{1}} \ar[d]^{g^{1}} & X ^{2}\ar[r]^{d_{X}^{2}} \ar[d]^{g^{2}}&  \cdots \ar[r] &  X ^{n} \ar[r]^{0}\ar[d]^{g^{n}}& 0 \ar[d]  \\
Z \ar[d]^{\widetilde{h}}& Z^{0} \ar[r]^{d_{Z}^{0}}\ar[d]^{0}& Z^{1} \ar[r]^{d_{Z}^{1}} \ar[d]^{h^{1}} & Z^2 \ar[r]^{d_{Z}^{2}}\ar[d]^{h^{2}}& \cdots \ar[r] & Z^{n} \ar[r]^{d_{Z}^{n}}\ar[d]^{h^{n}} & Z^{n+1}\ar[d]^{h^{1}}  \\
Y& 0 \ar[r]& Y^{1} \ar[r]^{d_{Y}^{1}}  & Y^2 \ar[r]^{d_{Y}^{2}}& \cdots \ar[r] & Y^{n} \ar[r]^{0} & 0. \\}
 \end{displaymath}
\end{rem}
\vspace{.1in}

Bellow, we adapt the definition given  in  \cite{CGP2} of when a complex can be extended to the left or to the right  to any additive  category $\mathcal{A}$.

\begin{defi} Let $X=(X^i, d^i)$ be a complex in  $\mathbf{C}_{[1,n]}(\mathcal{A})$.  We say that $X$
 can be extended to the left if $d^1$ is not a monomorphism and that $X$ can be extended to the right if there is an object $X'^{n+1} \in \mathcal{A}$ and a non-zero morphism $d'^{n}: X^{n} \rightarrow X'^{n+1}$ such that $d'^{n} d^{n-1}=0$.
\end{defi}

Now, we present this useful result.

\begin{teo} \label{irre en Cn}   Let $f:X \to Y$ be a morphism in  $\mathbf{C}_{[1,n]}(\mathcal{A})$  where $X$ can not be extended to the left  and $Y$  can not be extended to the right.
If $m\geq  n$  then, the following conditions are equivalent.
  \begin{enumerate}
\item $f$  is irreducible in ${\C^{b}(\mathcal{A})}$.
\item $f$ is irreducible in   $\mathbf{C}_{[0,n+1]}(\mathcal{A})$.
\item   ${\mathcal F}_m(f) $ is irreducible in $\mathbf{C}_{\equiv m}(\mathcal{A})$.
\item  $f$  is irreducible  in $\mathbf{C}_{[1,n]}(\mathcal{A})$.
\end{enumerate}
\end{teo}
\begin{proof} By Theorem \ref{irred} and Remark \ref{0n+1} we get the equivalence between the tree first statements. Moreover (4) follows if  $f$  is irreducible in ${\C^{b}(\mathcal{A})}$. We observe that to prove that Statement(1) is equivalent with Statement(4), the hypothesis of $m\geq  n$ is not necessary.

Now,  if $f$  is irreducible  in $\mathbf{C}_{[1,n]}(\mathcal{A})$ and  factors trough a complex  $Z  \in \mathbf{C}_{[0,n+1]}(\mathcal{A})$ then we have the following situation:

\begin{displaymath}
\xymatrix {X  \ar[d]^{\widetilde{g}} & 0 \ar[r] \ar[d]&  X ^{1} \ar[r]^{d_{X}^{1}} \ar[d]^{g^{1}} & X ^{2}\ar[r]^{d_{X}^{2}} \ar[d]^{g^{2}}&  \cdots \ar[r] &  X ^{n} \ar[r]^{0}\ar[d]^{g^{n}}& 0 \ar[d]  \\
Z \ar[d]^{\widetilde{h}}& Z^{0} \ar[r]^{d_{Z}^{0}}\ar[d]^{0}& Z^{1} \ar[r]^{d_{Z}^{1}} \ar[d]^{h^{1}} & Z^2 \ar[r]^{d_{Z}^{2}}\ar[d]^{h^{2}}& \cdots \ar[r] & Z^{n} \ar[r]^{d_{Z}^{n}}\ar[d]^{h^{n}} & Z^{n+1}\ar[d]^{h^{1}}  \\
Y& 0 \ar[r]& Y^{1} \ar[r]^{d_{Y}^{1}}  & Y^2 \ar[r]^{d_{Y}^{2}}& \cdots \ar[r] & Y^{n} \ar[r]^{0} & 0 \\}
 \end{displaymath}

\noindent where either $\{g^1, g^2, \dots , g^n\}$ is a section or $\{h^1, h^2, \dots , h^n\}$ is a retraction in $\mathbf{C}_{[1,n]}(\mathcal{A}).$

In the former case, since $d_X^1$ is a monomorphism then we get that ${\widetilde{g}}$ is a section in $ \mathbf{C}_{[0,n+1]}(\mathcal{A}).$  In the latter case, we consider ${\overline{h}}^n :Y^n \rightarrow Z^n$ the morphism such that  $h^n{\overline{h}}^n=id.$ Then we have that $d_Z^{n}{\overline{h}}^n=0$. In fact, otherwise $d_Z^{n}{\overline{h}}^nd_Y^{n-1}=0$ and therefore  $Y$ can be extended to the right getting a contradiction with the hypothesis. Thus, we  conclude that ${\widetilde{h}}$ is a retraction in $ \mathbf{C}_{[0,n+1]}(\mathcal{A}).$
\end{proof}

\subsection{Density of the compression functor}
In this section, we show that the density is an invariant under stable categories. If $\mathcal{A}$ is an additive category, by \cite[Corollary 3.4]{Sa} we know that  $\C_{\equiv m}( \A)$  is a Frobenius category.
\vspace{.05in}

Now, we give the projective-injective $m-$periodic objects, first studied in \cite{B} and \cite{Z} for particular cases. The proof of this result follows from  Lemma 3.6  in \cite{Sa}, with a similar proof than the one done by D. Happel in \cite{H}.
\vspace{.05in}

\begin{prop} \label{K_P} Let $\mathcal A$ be an additive category and $m\geq 2$.  The projective-injective objects in the Frobenius category ${\C_{\equiv m}({\mathcal A})}$   with the graded split exact structure are, up to shift, the $m-$periodic objects  $K_M:  \dots \rightarrow M \stackrel{id_M}\rightarrow M  \stackrel{0}\rightarrow 0 \stackrel{m-2}\rightarrow 0 \rightarrow M \stackrel{id_M}\rightarrow M \rightarrow \dots$ given by  each  object $M \in \mathcal{A}.$
\end{prop}

As a consequence we have the following result.

\begin{teo} \label{stdense1} Let $\A$ be an additive category and $m\geq 2.$

The compression functor ${\mathcal F}_m: \C^b (\mathcal{A}) \rightarrow  \C_{\equiv m} (\mathcal{A})$ is dense if and only if  the compression functor
 ${\mathcal F}_m: \underline {\C^b (\mathcal{A}) }\rightarrow  \underline{\C_{\equiv m} (\mathcal{A})}$ so is.
\end{teo}
\begin{proof} By Proposition  \ref{K_P} and Proposition \ref{m>n}, we know that  if $M\in \mathcal{A}$  and $X:   M \stackrel{id_M}\rightarrow M  \in \C_{[1,2]} (\mathcal{A})$ then ${\mathcal F}_m(X)=K_M,$ and moreover ${\mathcal F}_m(X[i])=K_M[i]$ for $i=0,\dots, m-1$.
 Then the projective-injective objects in the Frobenius category ${ \C_{\equiv m}(\mathcal{A})}$   with the graded split exact structure correspond to the projective-injective objects in the Frobenius category ${\C^{b}(\mathcal{A})}$   with the graded split exact structure. Then we conclude the result.
 \end{proof}

\section{The Galois $G-$covering ${\mathcal F}_m$ }

Let $\A$ be an additive Hom-finite Krull-Schmidt category. Throughout this section, $\A$ verifies the following condition:

There is a positive integer $n\geq 2$ such that each indecomposable in $\C^{b}(\A)$ is a shift of an indecomposable in $\C_{[1,n]}(\A).$
\vspace{.1in}

Consider the compression functor ${\mathcal F}_m: \C^b (\mathcal{A}) \rightarrow  \C_{\equiv m} (\mathcal{A})$.

\begin{teo} \label{dense}
For every indecomposable $Z\in \C_{\equiv m} (\A)$ there is an indecomposable $\widehat{Z}\in \C_{[1,n]}(\A)$ and  an integer  $t \in \{0, \dots, m-1\}$ such that $Z= {\mathcal F}_{m}(\widehat{Z}[t]).$

In particular, we have ${\mathcal F}_m$ is dense.
\end{teo}
\begin{proof} First, consider  $m \geq n$ and $Z$ an indecomposable complex in $\C_{\equiv m} (\mathcal A).$
If  for some integer $i \in \{1, \dots, m\}$ we have that $d_{Z}^{i}=0$ then we consider
 $i$ to be the biggest integer such that $Z^{i}\neq 0$ and $d_Z^i=0.$

Now, consider   the complex
$$\widehat{Z} :  \, Z^{i+1} \rightarrow  \dots \rightarrow Z^{m}\rightarrow Z^{1}\rightarrow \dots \rightarrow Z^{i}  \in \C_{[1,m]}(\A)$$

\noindent  Therefore, by Proposition \ref{m>n},
 $Z=  {\mathcal F}_{m}(\widehat{Z}[t])$  with $t=m-i  \in \{0, \dots, m-1\}.$ From the fact that $Z$ is indecomposable   we get that $\widehat{Z}$ is  indecomposable in $\C^{b}(\mathcal A)$.  Then by the property of the objects in  $\C^{b}(\mathcal A)$, we have that $\widehat{Z} \in \C_{[1,n]}(\mathcal A),$ up to shifts.

Now, we shall see that it is not possible to have an indecomposable complex $Z \in \C_{\equiv m} (\A)$  such that all the differentials $d_{Z}^{i}$  are non-zero.

Assume that there is  an $m-$periodic complex $Z$ as follows:

 $$Z : \dots \longrightarrow Z^{1}\stackrel{d_{Z}^{1}} \longrightarrow Z^{2}\stackrel{d_{Z}^{2}}\longrightarrow \dots \longrightarrow Z^{m}\stackrel{d_{Z}^{m}}\longrightarrow Z^{1}\stackrel{d_{Z}^{1}}\longrightarrow \dots$$

\noindent where all the differentials are non-zero.

Consider the complex in $\widehat{Z}  \in \C^{b}(\A)$ as follows:

 $$\widehat{Z} :0 \rightarrow Z^{1}\stackrel{d_{Z}^{1}} \longrightarrow Z^{2}\stackrel{d_{Z}^{2}}\longrightarrow \dots \longrightarrow Z^{m}\stackrel{d_{Z}^{m}}\longrightarrow Z^{m+1}\stackrel{d_{Z}^{1}} \longrightarrow   \dots \longrightarrow Z^{2m}\longrightarrow   Z^{1}\stackrel{d_{Z}^{2m+1}} \longrightarrow  \dots \longrightarrow Z^{3m}\rightarrow 0$$

\noindent where $Z^{2m+i}=Z^{m+i}=Z^i$ for $i=1, 2, \cdots, m$. Since $\widehat{Z}$ has $3m$ non-zero entries and    $3m \geq 3n \geq 2n + 2>n$ then $\widehat{Z}$ is not an indecomposable complex in $\C^{b}(\mathcal A)$. Furthermore, we must have an indecomposable direct summand $\widehat{X}$ of $\widehat{Z}$ having a zero in  the first entry  and also in the $3m-$entry, as follows:

$$\widehat{X}: 0 \rightarrow 0 \rightarrow \cdots X^i \stackrel{d_{X}^{i}}\longrightarrow X^{i+1}\stackrel{d_{X}^{i+1}}\longrightarrow \dots \longrightarrow X^{i+n}\stackrel{d_{X}^{i+n}}\longrightarrow 0 \cdots $$
\vspace{.1in}

\noindent with $1<i<n+2\leq 2n\leq 2m< 3m-n$ and $X^i \neq 0$. Then there is a section

{\tiny
\begin{displaymath}
\xymatrix { Z^{1} \ar[r]&  \cdots \ar[r] & Z^{i-1}\ar[r]  & Z^{i}\ar[r] & \cdots \ar[r] & Z^{n +i}  \ar[r] & \cdots\ar[r] & Z^{m+i-1}\ar[r] &\cdots \ar[r] & Z^{3m} \\
  0 \ar[r] \ar[u] & \cdots \ar[r] & 0 \ar[u]\ar[r]&  X^{i}\ar[u]\ar[r] & \cdots \ar[r] & X^{n +i}\ar[u]  \ar[r]& \cdots\ar[r]   & 0 \ar[u]\ar[r]&\cdots \ar[r] & 0\ar[u] \\}
 \end{displaymath}
 }

\noindent in $\C^{b}(\A)$, and as we show below, we get that  $\mathcal{F}_m(\widehat{X})$ is a direct summand of ${Z}$ in $\mathbf{C}_{\equiv m} (\A)$.

{\tiny{
\begin{displaymath}
\xymatrix {  Z^{1} \ar[r]&  \cdots \ar[r]& Z^{i+n-j+1} \ar[r] & Z^{i+n-j+2} \ar[r]& \cdots \ar[r]&Z^{i-1}  \ar[r] & Z^{i}  \ar[r]& \cdots \ar[r] & Z^1 \\
             X^j \ar[r] \ar[u]  & \cdots\ar[r] & X^{i+m} \ar[r]\ar[u] &0  \ar[r]\ar[u]& \cdots \ar[r]&0  \ar[r]\ar[u] & X^i \ar[r]\ar[u]& \cdots \ar[r] & X^j\ar[u] \\}
 \end{displaymath} }}
\vspace{.1in}

\noindent getting a contradiction to the fact that ${Z}$ is indecomposable. Therefore, we prove that  for some $i \in \{1, \dots, m \}$ we have  that $d_{Z}^{i}=0$. Therefore, we get that the functor  $\mathcal{F}_m$ is dense, whenever $m \geq n$.

Now, assume that $m< n$ and consider $w$  the positive integer such that $ w=\mbox{min} \{s \geq 1, \, \mbox{such that }   ms \geq n \}.$

From the above arguments  and using that $mw \geq n,$ we have that  ${\mathcal F}_{m w}: \C^b (\mathcal{A}) \rightarrow  \C_{\equiv m w} (\mathcal{A})$ is dense. Consider  $Y\in  \C_{\equiv m} (\mathcal{A})$ an indecomposable $m-$complex. Observe that every $m-$periodic complex is in fact a $mw-$periodic complex.
Therefore, there is $Z \in  \C^b (\mathcal{A}) $ such that $Y={\mathcal F}_{m w}(Z).$
By Lemma \ref{descompos} we have that ${\mathcal F}_m(Z) =  \oplus_ {j\in [0,w-1] } {\mathcal F}_{mw}(Z)[mj] $. Then  $Y$ is a direct summand of   ${\mathcal F}_m(Z)$  in $\C^b (\mathcal{A}).$

Consider $Z=  \oplus_ {j\in L } Z_j$   the indecomposable decomposition  of $Z$ in $\C^b (\mathcal{A}).$  Using that  ${\mathcal F}_m$ is additive, we have that ${\mathcal F}_m(Z) =  \oplus_ {j\in L } {\mathcal F}_m(Z_j).$ Moreover, for each $ {j\in L }$ we know that ${\mathcal F}_m(Z_j)$ is indecomposable in $ \C_{\equiv m} (\mathcal{A})$ because of Proposition \ref{ind}. Then $Y={\mathcal F}_m(Z_j)$ for some $j\in L$  since  $\mathcal{A}$ is a Krull-Schmidt category.
Using that $Y$ is $m-$periodic and the property of the indecomposable complexes in $\C^{b}(\mathcal A)$ we infer that there is a $\widehat{Z_j} \in \C_{[1,n]}(\mathcal A)$ such that $Z_j=\widehat{Z_j} [t]$ for some $t \in \{0, \dots, m-1\}.$
Then we conclude the result.
\end{proof}

Now, we are in position to prove that ${\mathcal F}_m$ is  a Galois $G-$covering for  each $m.$

 \begin{teo}  \label{Gcov}
The compression functor
   ${\mathcal F}_m: \C^b (\mathcal{A}) \rightarrow  \C_{\equiv m} (\mathcal{A})$ is  a Galois $G-$covering in the sense of \cite{BL},   for  each $m.$
\end{teo}
\begin{proof} The functor  ${\mathcal F}_m: \C^b (\mathcal{A}) \rightarrow  \C_{\equiv m} (\mathcal{A})$ is almost dense since by   Theorem \ref{dense} we have that the functor ${\mathcal F}_{m}$ is dense.

On the other hand, if $X \in  \C^{b}(\A)$ is indecomposable, then by Proposition \ref{ind}, we know that  ${\mathcal F}_{m}(X)$ is indecomposable in $\mathbf{C}_{\equiv m} (\A)$.

Consider  $X, Y\in  \A$  indecomposable complexes  in $\C^{b}(\A)$ with ${\mathcal F}_{m}(X) = {\mathcal F}_{m}(Y)$. If $X$ is an  indecomposable complex  in $\C^{b}(\A)$ then $X[mj]$ either so is. Since ${\mathcal F}_{m}(X)  = {\mathcal F}_{m}(Y)$ then we have that

$$ {\mathcal F}_{m}(X) = \oplus_{j \in \mathbb{Z}} X[m j] =  \oplus_{j \in \mathbb{Z}} Y[m j] = {\mathcal F}_{m}(Y).$$

\noindent Then there exist integers $j_{0}$ and $j_1$  such that $X [m j_{0}] = Y[m j_{1}]$. Therefore, $X = Y[m(j_{0}- j_1]$,  proving that ${\mathcal F}_{m}$ is a Galois $G$-covering.
\end{proof}

As an immediate consequence of Theorem \ref{Gcov} we get the following corollary.

\begin{coro}
The functor ${\mathcal F}_m $   gives a bijective map between the sets $\emph{Ind} (\C_{\equiv m}(\A))$  and
$\cup_{k=0}^{m-1} \emph{Ind} (\C_{[1,n]}(\A)[k]).$
\end{coro}
\begin{proof} The bijective map follows from the fact that there is only one indecomposable if we take a shift in the interval $ [0,m-1].$
\end{proof}

We recall the definition of strong global dimension given in \cite{Sk}, for a finite dimensional algebra over a field.

\begin{defi} Let $A$ be a finite dimensional algebra over a field. 
The strong global dimension of $A$ denoted by $s.gl. \emph{dim}\, A$ is 
$\emph{sup}\{\, \ell(X) \mid X \in  K^{b}(\emph{proj}\, A) \text{\, is indecomposable}\}.$
\end{defi}

The previous results,    allow us to state the next result for a finite dimensional $k$-algebra $A$ with finite strong global dimension, since $\A= \mbox{proj}\, A$ satisfies the property stated at the beginning of Section 3.
\vspace{.05in}

In general,  for  a finite dimensional $k$-algebra  with $s. gl. \mbox{dim}\, A = \infty$ the compression functor ${\mathcal F}_m: \C^b (\mbox{proj}\,A)\rightarrow \mathbf{C}_{\equiv m} (\mbox{proj}\,A)$  is not dense, as we can see in \cite[Proposition 5.4 and  Theorem 5.5]{S-18}.

\begin{teo} \label{dense1} If $A$ is a  finite dimensional $k$-algebra over a field $k$ with  $s.gl.\emph{dim} \, A =\nu$  then the following statements hold.
\begin{enumerate}
\item The compression functor  ${\mathcal F}_m: \C^b (\emph{proj}\, A) \rightarrow   \C_{\equiv m}(\emph{proj}\, A)$ is a Galois $G$-covering in the sense of \cite{BL}.
 \item   There is a bijective correspondence between the sets $\, \emph{Ind}( \C_{\equiv m}(\emph{proj} \, A) )\, $ and
 \newline $\bigcup _{k=0}^{m-1} \emph{Ind}( \C_{[1, \nu +1]}(\emph{proj} \, A)[k])$.
\end{enumerate}
\end{teo}
\vspace{.05in}

\subsection{Dualizing $k-$categories}

Dualizing $k-$categories were introduced by M. Auslander and I. Reiten in \cite{AR} as a generalization of artin $k$-algebras.

Let $\A$ be a Hom-finite Krull-Schmidt $k$-category. By $\mbox{Mod}\,  \A$ we denote the category of contravariant functors $F : \A \rightarrow \mathrm{Mod} \, k$ and by  $\mbox{mod}\, \A$  the full subcategory of $\mbox{Mod}\,  \A$ whose objects are the finitely presented functors.

The category $\mathcal{A}$ is called dualizing  if the finitely presented objects and the
 finitely copresented objects in $\mathrm{Mod}\,\mathcal{A}$ coincide, see \cite{AR} and \cite{G}.
If $\mathcal{A}$ is  dualizing, $\mathrm{mod} \, \mathcal{A}$ is an abelian category with enough projective objects. Therefore  the category $\mathbf{C}^{b}(\mathrm{mod} \,\mathcal{A})$ is also abelian.

From \cite{BSZ} we know that  if $\mathcal{A}$ is a dualizing category, then the category $\mathbf{C}^{b}(\mathrm{mod}\, \mathcal{A})$ so is and the usual exact structure $\mathcal{F}$ on $\mathbf{C}^{b}(\mathrm{mod}\, \mathcal{A})$ has almost split sequences. Moreover, $\mathbf{C}_{[1,n]}(\mathcal{A})$ has almost split sequences. If $\mathrm{mod}\, \mathcal{A}$  has finite global dimension then $\mathbf{C}^{b}( \mathcal{A})$  has almost split sequences, see for example \cite{HZ}.
\vspace{.1in}

In all that follows,  $\A$ shall be a  Hom-finite Krull-Schmidt $k-$category
 verifying that there is a positive integer $n\geq 2$ such that each indecomposable complex in $\C^{b}(\A)$  is a shift of an indecomposable  complex in $\C_{[1,n]}(\A).$

\begin{teo}
Let $\A$ be a  dualizing category such that $\mathrm{mod}\, \mathcal{A}$ has finite global dimension. Then $ \C_{\equiv m} (\mathcal{A})$ has almost split sequences.
\end{teo}
\begin{proof} By Theorem \ref{Gcov}, we know that  ${\mathcal F}_{m}$ is a Galois $G$-covering. Moreover, by \cite[Corollary 2]{HZ} if $\mathrm{mod}\, \mathcal{A}$  has finite global dimension then $\mathbf{C}^{b}( \mathcal{A})$  has almost split sequences.

By  \cite[Theorem 3.7]{BL} we know that a short sequence $\eta$ in $\C^{b}( \A)$ is an almost split sequence  if and only if ${\mathcal F}_{m}(\eta)$ is an almost split sequence in $\mathbf{C}_{\equiv m}(\mathcal{A})$.
 \end{proof}

\begin{rem}\label{nota} \emph{ In \cite[Corollary 5]{HZ},  the authors proved that if $\A$ is   dualizing
  and }$\mathrm{mod}\, \mathcal{A}$  has finite global dimension then $\mathbf{C}^b_{\equiv m}(\mathcal{A})$ has almost split sequences. Here $\mathbf{C}^b_{\equiv m}(\mathcal{A})$ denotes the  smallest full subcategory of $\mathbf{C}_{\equiv m}(\mathcal{A})$ containing all the stalk $m-$periodic complexes closed under finite extensions.

\emph{The mentioned authors proved that when $\mbox{mod}\, \mathcal{A}$ has global dimension less than  or equal to one then both categories $\mathbf{C}_{\equiv m}(\mathcal{A})$ and $\mathbf{C}^b_{\equiv m}(\mathcal{A})$  coincide.}
\end{rem}
\vspace{.05in}

By \cite{BSZ}, we know that there are  almost split sequences in $\mathbf{C}_{[1, n]}(\A),$ whenever $\A$ is  dualizing. The next result shows that these sequences provide the ones in $\mathbf{C}_{\equiv m}(\A).$

\begin{teo} \label{nu}  Let $\A$ be a dualizing $k-$category such that $\mbox{mod}\, \mathcal{A}$ has finite global dimension. Then each almost split sequence in $\mathbf{C}_{\equiv m}( \A)$ is the image under $\mathcal{F}_m$ of the shift of an almost split sequence in $\mathbf{C}_{[1, n]}(\A)$.
\end{teo}

\begin{proof}  Let $\alpha:  A' \stackrel{f}{\rightarrow} B' \stackrel{g}{\rightarrow} C'$ be an almost split sequence in $\mathbf{C}_{\equiv m}(\A)$. By
Theorem \ref{dense} there is a complex  $X \in \C_{[1,n]}(\A)$ such that $\mathcal{F}_m(X[t])=A'$ for some integer $t.$

By Proposition \ref{K_P}, since $A'$ is not projective-injective  then $X$ neither so is.  By Theorem \ref{Gcov} we know that the functor $\mathcal{F}_m: \C^{b}(\A) \rightarrow \mathbf{C}_{\equiv m}(\A)$ is a  Galois $G$-covering thus we may consider $\eta: \ X \stackrel{h}{\rightarrow} Y \stackrel{l}{\rightarrow} Z$ an almost split sequence in $\C^{b}(\A)$ starting in $X$ such that ${\mathcal F}_{m}(\eta [t])=\alpha$.  Observe that we can consider the shift  such that $X^1\neq 0.$ Then $X^{n+1}=0$ because $X$ is indecomposable. Hence $X \in \C_{[1,n]}(\A).$

Now, we prove that $\eta $ is in fact an almost split sequence in $\mathbf{C}_{[1, n]}( \A).$

First, we affirm that $Y^{n+1}=0.$ In fact, consider $f:X \rightarrow  Y=\oplus_{i=1}^{r} Y_{i}$. We observe that if $Y^{n+1}\neq 0$ then there is an indecomposable direct summand of $Y,$ namely $Y_t$ such that $Y_t^{n+1}\neq 0$ and $f_t:X\rightarrow Y_t$ is an irreducible morphism between indecomposable complexes. We point out that $Y_t$ is not projective-injective,  otherwise $Y_t=J_n(P)$ and then $f=0$ because $id_Pf^n=0.$
 Then, the mapping cone $C_{f_t}$ is an indecomposable complex. Therefore, we get a contradiction because $C_{f_t}^{1}=X[1]^{1}\oplus Y_t^{1}=X^2\neq 0$ and $C_{f_t}^{n+1}=X[1]^{n+1}\oplus Y_t^{n+1}\neq 0.$ Then we conclude that $Y^{n+1}=0.$

From the fact that  $Y^{n+1}=X^{n+1} \oplus Z^{n+1}=Z^{n+1}$ we get that $Z^{n+1}= 0$.

Therefore we only need to show that $Y$ has length at most $n$ so  we can conclude that $X,Y,Z \in \mathbf{C}_{[1, n]}( \A).$
Assume that $Y^0\neq 0$, then there exists a indecomposable direct summand of $Y$, $Y_r$,  with $Y^0_r\neq 0$ and there also be  an irreducible morphism $f_r:X \to Y_r$.
Again, since an indecomposable complex has at most $n$ non-zero entries and $Y^0_r\neq 0$ we have that $Y_r^n=0$. Now, since $f_r$ is an irreducible morphism and $X^0=0$, then by \cite[Corollary 2, Proposition 3]{GM} the irreducible morphism $f_r$ is a split monomorphism and thus $X^n=0$.

Now, if this is the case we have that $Y^{0} =Z^0 \neq 0$ implies $Z^n=0$ and thus $Y^n=0$ because $Z$ is an indecomposable object.  Hence if $Y^{-1}=0$ we can choose another interval $[1, n]$ by shifting the complexes one entry to the left and we are done. If $Y^{-1}\neq 0$, we repeat the above argument to get that $X^{n-1}, Y^{n-1}$ and $Z^{n-1}$ are zero and we proceed to analyze if $Y^{-2}=0$ or not. We repeat this argument,  until there exists $i<n$ such that $Y^{-i}=0$. Observe that since $X^{1}\neq 0$ this process stops.
Selecting the proper interval we get that $X,Y,Z \in \mathbf{C}_{[1, n]}(\A)$.
\end{proof}

\section{The Auslander-Reiten quiver of $\C_{\equiv m} (\mbox{proj}\, A)$ where the strong global dimension of $A$ is finite.}

Throughout this section,  we consider  $A$  a finite dimensional $k$-algebra over a field $k$ and with   $s.gl.\mbox{dim}\, A = \nu < \infty$.


By \cite[Proposition 2.5]{AR}, we know that $\rm{proj}\,A$ is dualizing.

The aim of this section is to show how to build the  Auslander-Reiten quiver of $\C_{\equiv m} (\rm{proj}\,A)$. By the result given in \cite[Theorem 4.1]{CGP3}, we  shall give two ways of constructing this quiver. First applying the compression functor $\mathcal{F}_m$  and secondly applying the same knitting technique used to obtain the Auslander-Reiten quiver of $\mathbf{C}_{[1,n]}(\rm{proj}\,A)$.
\vspace{.05in}

By Theorem \ref{dense1} we know that if  $A$ is  a finite dimensional $k$-algebra such that $s. gl. \mbox{dim}\, A <\infty$ then ${\mathcal F}_{m}: \C^{b}(\mbox{proj} \,A)\rightarrow \mathbf{C}_{\equiv m} (\mbox{proj} \,A)$ with $m \geq  2$ is a Galois $G$-covering where $G$  is the cyclic group generated by $[m]$.
\vspace{.05in}

By \cite[Proposition 3.3]{BL}, if $u : X\rightarrow Y$ is a morphism in $\C^{b}(\mbox{proj} \,A)$
with $X$ or $Y$ indecomposable, then $u$ is irreducible in $\C^{b}(\mbox{proj} \,A)$ if and only if ${\mathcal F}_{m}(u)$ is irreducible
in $\mathbf{C}_{\equiv m}(\mbox{proj} \,A)$.

Furthermore, by \cite[Proposition 3.4]{BL} if $u : X\rightarrow Y$ is a morphism in $\C^{b}(\mbox{proj} \,A)$,
then $u$ is a minimal left  almost split morphism  or a minimal right  almost split morphism if and only if ${\mathcal F}_{m}(u)$ is a minimal left  almost split morphism
or a minimal right  almost split morphism  in $\mathbf{C}_{\equiv m}(\mbox{proj} \,A)$, respectively.}
\vspace{.1in}

As a consequence of  \cite[Theorem 3.7]{BL}, we get the following result  for $\A=\mbox{proj} \,A$ and  $m \geq  2$.

\begin{teo} \label{ass-conv}  Let $A$ be a finite dimensional $k$-algebra such that $s. gl. \emph{dim}\, A = \nu$.
The sequence $\eta$ in $\C^{b}(\emph{proj} \,A)$ is almost split  if and only if ${\mathcal F}_{m}(\eta)$ is an almost split sequence in $\mathbf{C}_{\equiv m}({\emph{proj}}\,A)$  for  any $m\geq 2$.

Moreover, the additive function can be applied to build the Auslander-Reiten quiver of $\C_{\equiv m} (\rm{proj}\,A)$.
\end{teo}
\begin{proof} By  \cite[Theorem 3.7]{BL} we have that $\eta$ in $\C^{b}(\mbox{proj} \,A)$ is almost split sequence  if and only if ${\mathcal F}_{m}(\eta)$ is an almost split sequence in $\mathbf{C}_{\equiv m}({\mbox{proj}}\,A)$  for  any $m\geq 2$.

On the other hand, $\mathcal{F}_m: \mathbf{C}^{b}(\mbox{proj} \,A) \rightarrow \mathbf{C}_{\equiv m}({\mbox{proj}}\,A)$   is a Galois $G$-covering, where $G$ is the cyclic group generated by $[m]$.

In the category of complexes   $\mathbf{C}^{b}({\rm proj} \,A)$ the  exact sequences are the ones of the form $\eta: X \stackrel{f}{\rightarrow} Y \stackrel{g}{\rightarrow} Z$ such that for each $i \in \mathbb{Z}$ the sequence $X^i \stackrel{f^i}{\rightarrow} Y^i \stackrel{g^i}{\rightarrow} Z^i$ is exact and splits. Then $Y^{i}= X^{i} \oplus Z^{i}$ for each $i \in \mathbb{Z}$.

Consider the almost split sequence $\mathcal{F}_m(\eta): \mathcal{F}_m( X) \rightarrow \mathcal{F}_m( Y) \rightarrow \mathcal{F}_m(Z)$ in $\mathbf{C}_{\equiv m}({\mbox{proj}}\,A)$. We have to prove that $\mathcal{F}_m(Y)^i = \mathcal{F}_m(X)^i \oplus \mathcal{F}_m(Z)^i$ for each $i$.

Hence

$$ \begin{array}{lcl}
  \mathcal{F}_m(X)^i \oplus \mathcal{F}_m(Z)^i & =&(\oplus_{t \in \mathbb{Z}}  X ^{i+tm}) \oplus (\oplus_{t \in \mathbb{Z}}  Z ^{i+tm})\\
   &=&\oplus_{t \in \mathbb{Z}}  (X ^{i+tm} \oplus Z ^{i+tm})\\
   &=& \oplus_{t \in \mathbb{Z}}  Y ^{i+tm} \\
   &=&  \mathcal{F}_m(Y)^i
\end{array}
$$

\noindent proving the result.
\end{proof}

Consider $\A= \mbox{proj}\, A$ with $A$ a finite dimensional $k$-algebra with   $s.gl.\mbox{dim}\, A = \nu$. Then $\A$ is a  Hom-finite Krull-Schmidt $k-$category verifying that each indecomposable complex in $\C^{b}(\mbox{proj}\, A)$  is a shift of an indecomposable  complex in $\C_{[1,\nu +1]}(\mbox{proj}\, A).$
\vspace{.05in}

By Theorem \ref{nu} we get the following result, which is fundamental to build the Auslander-Reiten quiver of $\C_{\equiv m} (\rm{proj}\,A)$ for $m \geq 2$.

\begin{teo} \label{nu2} Let $A$ be a finite dimensional $k$-algebra  such that $s. gl. \emph{dim}\, A = \nu$. Let  $m \geq 2$. Consider $\mathcal{F}_m: \C^{b}(\emph{proj} \, A) \rightarrow \mathbf{C}_{\equiv m}(\emph{proj} \, A)$ the compression el functor.  Then each almost split sequence in $\mathbf{C}_{\equiv m}(\emph{proj} \, A)$ is the image under $\mathcal{F}_m$ of the shift of an almost split sequence in $\mathbf{C}_{[1, \nu+1]}(\emph{proj} \, A)$.
\end{teo}

Now, we show how to build the Auslander-Reiten quiver of $\mathbf{C}_{\equiv m} (\mbox{proj} \,A)$, provided we know how to construct the Auslander-Reiten quiver of $\mathbf{C}_{[1,\nu +1]}({\rm proj} \,A)$. For more details of how to construct the Auslander-Reiten quiver of $\mathbf{C}_{[1,n]}({\rm proj} \,A)$ we refer the reader to \cite[Section 5]{CPS1}. We observe that if $\mathbf{C}_{[1,n]}({\rm proj} \,A)$ is of finite type then the technique given in \cite[Section 5]{CPS1} allow us to construct the whole quiver.

We have two ways of constructing the Auslander-Reiten quiver of $\mathbf{C}_{\equiv m} (\mbox{proj} \,A)$.

By  \cite[Theorem 4.1]{CGP2} since $A$ is a finite dimensional  piecewise hereditary algebra then we know that all the almost split sequences in $\C^{b}(\mbox{proj} \,A)$ are the ones included in the subquiver  $\Gamma(\Sigma)$ between the  section $\Sigma$ and $\Sigma[1]$  and containing both sections. By $\Sigma[1]$ we mean a shift of $\Sigma$.  More precisely, we can  find that situation in the quivers of the categories $\mathbf{C}_{[1,n]}({\rm proj} \,A)$. In particular, in a category $\mathbf{C}_{[1, \nu+1]}({\rm proj} \,A)$, where $s.gl.\mbox{dim}\, A  =\nu$.
\vspace{.1in}

For the first way of constructing the Auslander-Reiten quiver of $\mathbf{C}_{\equiv m} (\mbox{proj} \,A)$ we can follow the next steps.
\begin{itemize}
  \item[(1)] We construct the   Auslander-Reiten quiver of  $\mathbf{C}_{[1,\nu +1]}({\rm proj} \,A)$.

  \item[(2)] We consider a section $\Sigma$ and the subquiver $\Gamma(\Sigma)$ in $\mathbf{C}_{[1, \nu+1]}({\rm proj} \,A)$.

 \item[(3)] Since  $\mathcal{F}_m(X)=\mathcal{F}_m(X[m])$ then it is enough to consider the almost split sequences in $\Gamma(\Sigma)$ in $\mathbf{C}_{[1, \nu+1]}({\rm proj} \,A)$ and the shifts of this almost split sequences in $\Gamma(\Sigma)[i]$ for $i=1, \dots, m-1$ in $\C^{b}(\mbox{proj} \,A)$.

 \item[(4)]  Applying the functor  $\mathcal{F}_m$ to the almost split sequences stated in step (3) we obtain the Auslander-Reiten quiver of $\mathbf{C}_{\equiv m} (\mbox{proj} \,A)$.
\end{itemize}
\vspace{.1in}

A second way of constructing the Auslander-Reiten quiver of $\mathbf{C}_{\equiv m} (\mbox{proj} \,A)$ is as follows:

We may apply the knitting technique because of Theorem \ref{ass-conv}.  This technique has been used to build the Auslander-Reiten quiver of a module category and also for the categories of fixed size $\mathbf{C}_{[1,n]}({\rm proj} \,A)$, see \cite [Section 5]{CGP2}.

First, we start this process with the same steps  (1) and (2) stated above and we continue with steps (3)-(5) as we explain below.
\begin{itemize}
  \item[(3)] We may apply the functor $\mathcal{F}_m$  to the almost split sequences in $\Gamma(\Sigma)$.

 \item[(4)] We compute the radical of the projective-injective complexes in $\mathbf{C}_{\equiv m} (\mbox{proj} \,A)$. We know that if $X$ is the starting of an almost split sequence in $\Gamma(\Sigma)$ with a direct summand projective-injective $Y$  in its middle term then  $\mathcal{F}_m(X)$ is the starting of the almost split sequence in $\mathbf{C}_{\equiv m} (\mbox{proj} \,A)$ that has a projective-injective complex,  $\mathcal{F}_m(Y)$, in its middle term. We call $\mathcal{F}_m(X)$ the radical of $\mathcal{F}_m(Y)$ and we denote it by ${\rm rad}\,\mathcal{F}_m(Y)$. Moreover,  $\mathcal{F}_m(X[i])$ is the starting of the almost split sequence in $\mathbf{C}_{\equiv m} (\mbox{proj} \,A)$ that has the projective-injective complex   $\mathcal{F}_m(Y[i])$ in its middle term, for $i=1, \dots, m-1$. In other words, ${\rm rad}\,\mathcal{F}_m(Y[i])=\mathcal{F}_m(X[i])$.

\item[(5)] We continue building the quiver  applying the knitting technique.
\end{itemize}
\vspace{.1in}

Bellow we present an example to show how we can apply both methods of constructing the Auslander-Reiten quiver.

\begin{ej}\label{ej1}
\emph{Consider the path algebra given by the quiver $Q_A$}

\begin{displaymath}
\xymatrix {1 \ar[r]^{\alpha}  & 2 \ar[r]^{\beta} & 3 }
\end{displaymath}

\noindent \emph{with the relation $\beta \alpha  =0$. We observe that $s. gl. \mbox{dim}\, A  =2$.}

\emph{The Auslander-Reiten quiver of  $\mathbf{C}_{[1,3]}({\rm proj} \,A)$ is as follows:}
\vspace{.1in}

{\tiny
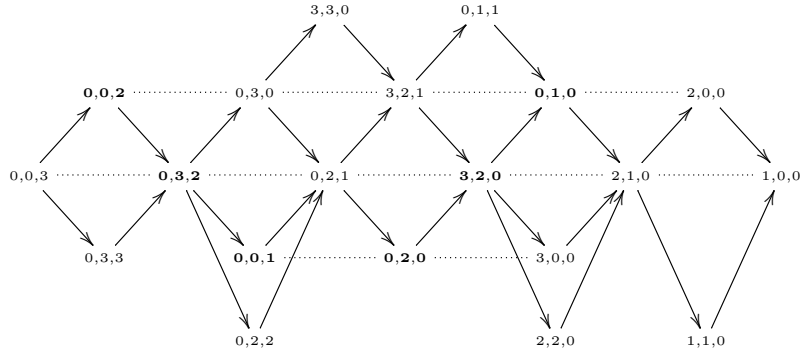
\begin{displaymath}
\def\objectstyle{\scriptstyle}
\def\labelstyle{\scriptstyle}
   \xymatrix @!0 @R=1.1cm  @C=1cm{ & & & & 3,3,0 \ar[rd]& &0,1,1\ar[rd] & && &   \\
    &{\bf  0 ,0,2}\ar@{.}[rr] \ar[rd] &  & { 0,3 ,0} \ar[rd]\ar[ru]\ar@{.}[rr]
   & & 3,2,1 \ar[rd] \ar[ru]\ar@{.}[rr]& & {\bf 0,1,0 }\ar[rd]\ar@{.}[rr]& & 2,0,0 \ar[rd] &  \\
  0 , 0,3 \ar[rd] \ar[ru] \ar@{.}[rr]& & \mathbf{0 ,3,2}\ar[rd] \ar[ru] \ar[rdd]\ar@{.}[rr] &&0,2,1 \ar[rd] \ar[ru]\ar@{.}[rr]& & \mathbf{3,2,0 }\ar[rd] \ar[ru] \ar@{.}[rr] \ar[rdd]
 &  &2,1,0 \ar[rdd]\ar[ru] \ar@{.}[rr] & & 1,0,0 \\
  &0,3,3 \ar[ru] & & \mathbf{0,0,1} \ar[ru] \ar@{.}[rr] & & \mathbf{0,2,0} \ar[ru]\ar@{.}[rr] & &3,0,0\ar[ru] & &  & \\
  & &    &  0,2,2 \ar[ruu] & & & &2,2,0  \ar[ruu]   & & 1,1,0 \ar[ruu]& }
\end{displaymath}}

\noindent \emph{where we denote the arrows in the complexes  with comas  and the projective $P_i$  by $i$.}

\emph{By \cite[Theorem 4.1]{CGP2} we know that all the almost split sequences in $\C^{b}(\mbox{proj} \,A)$ are the ones between a section $\Sigma$ and $\Sigma[1]$.   Furthermore, we may consider that  the complexes  $ 0\rightarrow 0\rightarrow P_2$, $0 \rightarrow P_3\rightarrow P_2$ and  $ 0\rightarrow 0\rightarrow P_1$ determine the section $\Sigma$.}

\emph{Applying $\mathcal{F}_4$ to the above almost split sequences in $\C^{b}(\mbox{proj} \,A)$ we obtain a subquiver of the Auslander-Reiten quiver of
$\mathbf{C}_{\equiv 4} (\mbox{proj} \,A)$ a follows:}

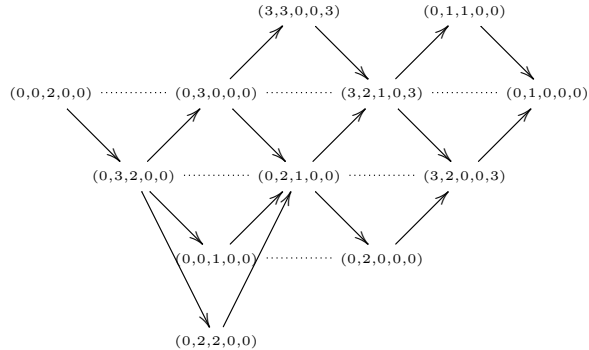
\tiny
\begin{displaymath}
\def\objectstyle{\scriptstyle}
\def\labelstyle{\scriptstyle}
   \xymatrix @!0 @R=1.1cm  @C=1.1cm{  & & &(3,3,0 ,0,3)\ar[rd] & &(0,1,1,0,0)\ar[rd] & \\
   (0,0,2,0,0)\ar@{.}[rr] \ar[rd]  & &( 0,3 ,0,0,0)\ar@{.}[rr] \ar[rd] \ar[ru]& &(3,2,1,0,3)\ar@{.}[rr] \ar[rd] \ar[ru] & &(0,1,0,0,0)  \\
    &(0,3,2,0,0)\ar@{.}[rr]\ar[rd] \ar[rdd] \ar[ru] & &(0 ,2,1,0,0)\ar@{.}[rr] \ar[rd] \ar[ru] & &(3,2,0,0,3)\ar[ru] & \\
 & &(0,0,1,0,0)\ar@{.}[rr]  \ar[ru] & & (0,2,0,0,0)\ar[ru]& &\\
&&(0,2,2,0,0)\ar[ruu]&&&& }
\end{displaymath}
\normalsize

\noindent \emph{where we denote each $4-$periodic complex in brackets, comas instead of arrows and $i$ instead of the projective $P_i$.}

\emph{Consider the almost split sequences $\Gamma(\Sigma)[i]$ for $i=0, \dots, 3$ in $\C^{b}(\mbox{proj} \,A)$.  Applying the functor  $\mathcal{F}_4$ to this sequences we obtain the Auslander-Reiten quiver of $\mathbf{C}_{\equiv 4} (\mbox{proj} \,A)$.}

\emph{ Another way to construct the Auslander-Reiten quiver of $\mathbf{C}_{\equiv 4} (\mbox{proj} \,A)$ is to apply the knitting technique. We start building the Auslander-Reiten quiver of $\mathbf{C}_{\equiv 4} (\mbox{proj} \,A)$. For this process, we need to know the radical of the projective-injective complexes in $\mathbf{C}_{\equiv 4} (\mbox{proj} \,A)$, which are the shifts of the radical of the projective-injective complexes stated in the above quiver.}

\emph{ Now, we apply the  knitting technique to build the Auslander-Reiten quiver. We obtain all the quiver by identifying the complexes that appear in the rectangles.}

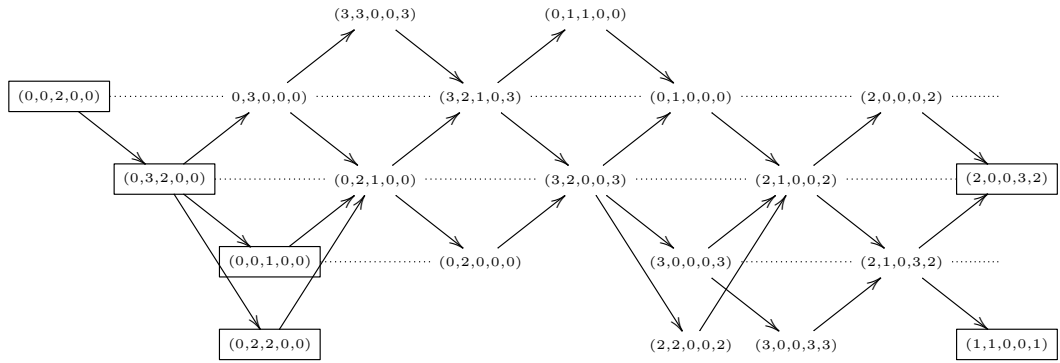
\tiny
\begin{displaymath}
\def\objectstyle{\scriptstyle}
\def\labelstyle{\scriptstyle}
   \xymatrix @!0 @R=1.1cm  @C=1.4cm{  & & &(3,3,0 ,0,3)\ar[rd] & &(0,1,1,0,0)\ar[rd] & & & &   \\
   *+[F]{(0,0,2,0,0)}\ar@{.}[rr] \ar[rd]  & & 0,3 ,0,0,0)\ar@{.}[rr] \ar[rd] \ar[ru]& &(3,2,1,0,3)\ar@{.}[rr] \ar[rd] \ar[ru] & &(0,1,0,0,0) \ar@{.}[rr]\ar[rd] & &(2,0,0,0,2)\ar@{.}[r]\ar[rd] &    \\
    &*+[F]{(0,3,2,0,0)}\ar@{.}[rr]\ar[rd] \ar[rdd] \ar[ru] & &(0 ,2,1,0,0)\ar@{.}[rr] \ar[rd] \ar[ru] & &(3,2,0,0,3)\ar@{.}[rr]\ar[rd] \ar[rdd] \ar[ru] & & (2,1,0,0,2)\ar@{.}[rr] \ar[rd] \ar[ru] & &*+[F]{(2,0,0,3,2)} \\
 & &*+[F]{(0,0,1,0,0)}\ar@{.}[rr]  \ar[ru] & & (0,2,0,0,0)\ar[ru]& & (3,0,0,0,3)\ar@{.}[rr]\ar[rd]\ar[ru] & & (2,1,0,3,2)\ar@{.}[r]\ar[rd]\ar[ru]& \\
&&*+[F]{(0,2,2,0,0)}\ar[ruu]&    &    && (2,2,0,0,2)\ar[ruu]&(3,0,0,3,3)\ar[ru]&&*+[F]{(1,1,0,0,1)}}
\end{displaymath}
\normalsize
\vspace{.1in}

\tiny
\begin{displaymath}
\def\objectstyle{\scriptstyle}
\def\labelstyle{\scriptstyle}
   
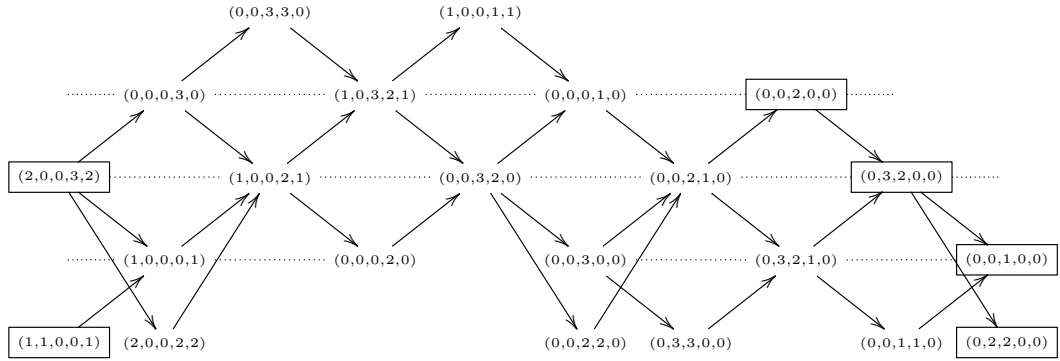
\xymatrix @!0 @R=1.1cm  @C=1.4cm{   & &(0,0,3 ,3,0)\ar[rd] & &(1,0,0,1,1)\ar[rd] & & & & &  \\
    &( 0,0 ,0,3,0)\ar@{.}[rr] \ar[rd] \ar[ru]\ar@{.}[l]& &(1,0,3,2,1)\ar@{.}[rr] \ar[rd] \ar[ru] & &(0,0,0,1,0) \ar@{.}[rr]\ar[rd] & &*+[F]{(0,0,2,0,0)}\ar@{.}[r]\ar[rd] &   & \\
    *+[F]{(2,0,0,3,2)}\ar@{.}[rr]\ar[rd] \ar[rdd] \ar[ru] & &(1 ,0,0,2,1)\ar@{.}[rr] \ar[rd] \ar[ru] & &(0,0,3,2,0)\ar@{.}[rr]\ar[rd] \ar[rdd] \ar[ru] & & (0,0,2,1,0) \ar@{.}[rr] \ar[rd] \ar[ru] & &*+[F]{(0,3,2,0,0)}\ar@{.}[r]\ar[rd] \ar[rdd] & \\
 &(1,0,0,0,1)\ar@{.}[l]\ar@{.}[rr]  \ar[ru] & & (0,0,0,2,0)\ar[ru]& & (0,0,3,0,0)\ar@{.}[rr]\ar[rd]\ar[ru] & & (0,3,2,1,0)\ar@{.}[rr]\ar[rd]\ar[ru]& &
 *+[F]{ (0,0,1,0,0)}\\
*+[F]{(1,1,0,0,1)} \ar[ru]&(2,0,0,2,2)\ar[ruu]&    &    && (0,0,2,2,0)\ar[ruu]&(0,3,3,0,0)\ar[ru]&&( 0,0,1,1,0)\ar[ru]& *+[F]{(0,2,2,0,0)} }
\end{displaymath}
\normalsize
\end{ej}
\vspace{.15in}

In the next example we show how to  construct the Auslander-Reiten quiver of $\mathbf{C}_{\equiv 2} (\mbox{proj} \,A)$ when the complexes belong to $\mathbf{C}_{[1,4]} (\mbox{proj} \,A)$.

\begin{ej}\label{ej2}

\emph{Consider the path algebra given by the quiver $Q_A$}

\begin{displaymath}
\xymatrix {1 \ar[r]^{\alpha}  & 2 \ar[r]^{\beta} & 3 \ar[r]^{\gamma}&4 }
\end{displaymath}

\noindent \emph{with the relations $\beta \alpha  =  \gamma \beta = 0$. We observe that $s. gl. \mbox{dim}\, A  =3$.}

\emph{We construct the Auslander-Reiten quiver of  $\mathbf{C}_{[1,4]}({\rm proj} \,A)$ until we find a section and the shift of the section as follows:}
\vspace{.1in}

{\tiny
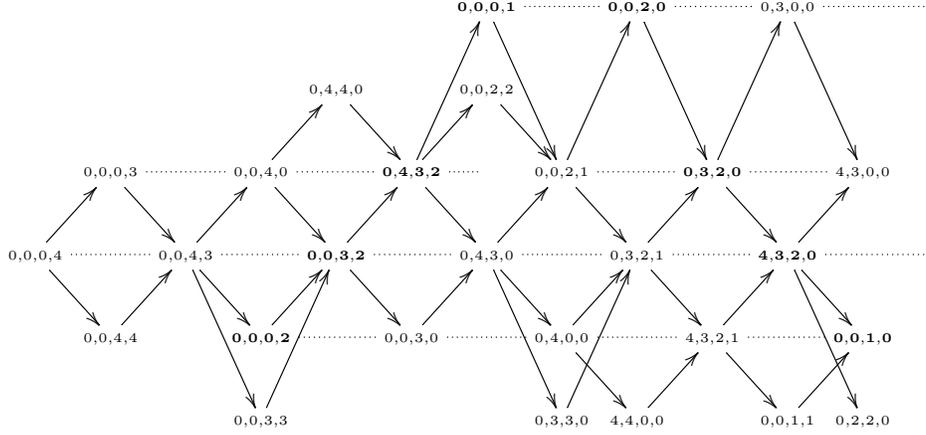
\begin{displaymath}
\def\objectstyle{\scriptstyle}
\def\labelstyle{\scriptstyle}
   \xymatrix @!0 @R=1.1cm  @C=1cm{ & & & & & &{\bf 0,0,0,1}\ar[rdd]\ar@{.}[rr] & & {\bf 0,0,2,0}\ar[rdd]\ar@{.}[rr] & &0,3,0,0 \ar[rdd]\ar@{.}[rr]& &   \\
    & &  &  & {  0 ,4,4,0} \ar[rd]&&0,0,2,2 \ar[rd] & & &  &  & & \\
&0,0,0,3 \ar[rd]  \ar@{.}[rr] &  &0,0,4,0\ar[rd] \ar[ru]\ar@{.}[rr] && {\bf 0,4,3,2} \ar[rd]\ar[ru]\ar[ruu]\ar@{.}[r] & &
0,0,2,1 \ar[rd]\ar[ruu]\ar@{.}[rr] &  & {\bf 0,3,2,0}\ar[rd]\ar[ruu]\ar@{.}[rr] &  & 4,3,0,0   &\\
0,0,0,4 \ar[rd] \ar[ru] \ar@{.}[rr] & & 0,0,4,3\ar[rd] \ar[ru]\ar[rdd] \ar@{.}[rr] &  &{\bf 0,0,3,2}\ar[rd] \ar[ru] \ar@{.}[rr] &  & 0,4,3,0\ar[rd] \ar[rdd]\ar[ru] \ar@{.}[rr] & &0,3,2,1 \ar[rd] \ar[ru] \ar@{.}[rr] &  &{\bf 4,3,2,0} \ar[rd] \ar[rdd]\ar[ru]\ar@{.}[rr] & &\\
&0,0,4,4 \ar[ru]& &{\bf 0,0,0,2}\ar[ru]\ar@{.}[rr] & & 0,0,3,0\ar[ru]\ar@{.}[rr]&  &0,4,0,0\ar[rd] \ar[ru]\ar@{.}[rr]& & 4,3,2,1\ar[rd] \ar[ru]\ar@{.}[rr] & &{\bf 0,0,1,0} &\\
  & &    & 0, 0,3,3 \ar[ruu] & & & &0,3,3,0  \ar[ruu]   & 4,4,0,0 \ar[ru] & &0,0,1,1 \ar[ru] & 0,2,2,0&}
\end{displaymath}}

 \emph{The complexes  $0\rightarrow 0\rightarrow 0\rightarrow P_2$, $0 \rightarrow 0 \rightarrow P_3\rightarrow P_2$,
$ 0\rightarrow P_4\rightarrow P_3 \rightarrow P_2$ and $0\rightarrow 0\rightarrow 0\rightarrow P_1$  determine the section $\Sigma$ that we need to construct the quiver of $\mathbf{C}_{\equiv 2}(\mbox{proj} \,A)$.}

\emph{ Applying the functor  $\mathcal{F}_2$ to the almost split sequences in $\C^{b}(\mbox{proj} \,A)$ that we obtained in the Auslander-Reiten quiver of $\mathbf{C}_{[1,4]} (\mbox{proj} \,A)$ between  $\Sigma$ and $\Sigma[1]$ and furthermore applying the knitting technique we obtain the Auslander-Reiten quiver of $\mathbf{C}_{\equiv 2}(\mbox{proj} \,A)$, which is the following:}

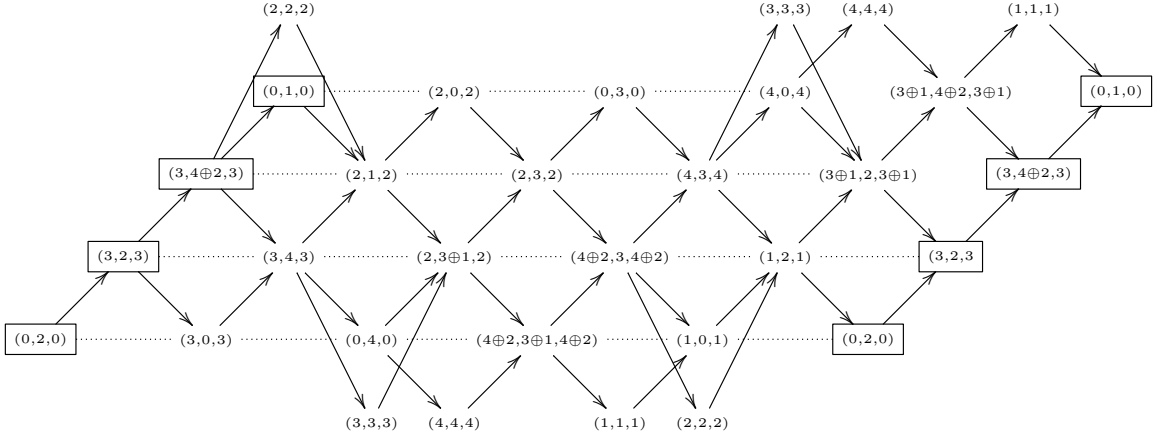
\tiny
\begin{displaymath}
\def\objectstyle{\scriptstyle}
\def\labelstyle{\scriptstyle}
   \xymatrix @!0 @R=1.1cm  @C=1.1cm{ & & &(2,2,2)\ar[rdd] & & & &&    &(3,3,3)\ar[rdd] & (4,4,4)\ar[rd]&& (1,1,1)\ar[rd]&\\
    & & &*+[F]{ (0,1,0)}\ar@{.}[rr] \ar[rd] & & (2,0,2)\ar@{.}[rr] \ar[rd] && (0,3,0)\ar@{.}[rr] \ar[rd] &&(4,0,4) \ar[rd]\ar[ru] &&(3\oplus1,4\oplus2,3\oplus1)\ar[rd] \ar[ru]& &*+[F]{(0,1,0)} \\
    & &*+[F]{(3,4\oplus2,3)}\ar@{.}[rr]\ar[rd] \ar[ruu] \ar[ru] & &(2,1,2)\ar@{.}[rr] \ar[rd] \ar[ru] & &(2,3,2)\ar@{.}[rr] \ar[rd] \ar[ru]&& (4,3,4) \ar@{.}[rr] \ar[rd] \ar[ru] \ar[ruu]&&(3\oplus1,2,3\oplus1)\ar[rd] \ar[ru]& &*+[F]{(3,4\oplus2,3)}\ar[ru] & & &\\
  &*+[F]{(3,2,3)}\ar@{.}[rr]  \ar[ru]\ar[rd] & & (3,4,3)\ar[ru]\ar[rd]\ar[rdd]\ar@{.}[rr]& &(2,3\oplus1,2)\ar@{.}[rr] \ar[rd] \ar[ru]&&(4\oplus2,3,4\oplus2)
 \ar[ru]\ar[rd]\ar[rdd]\ar@{.}[rr]&&(1,2,1)\ar@{.}[rr]  \ar[ru]\ar[rd] & &*+[F]{(3,2,3}\ar[ru]& &\\
*+[F]{(0,2,0)}\ar[ru]\ar@{.}[rr]&& (3,0,3)\ar[ru]\ar@{.}[rr] && (0,4,0)\ar[rd]\ar[ru]\ar@{.}[rr]  &&(4\oplus2,3\oplus1,4\oplus2)\ar[ru]\ar[rd]\ar@{.}[rr]&& (1,0,1) \ar[ru]\ar@{.}[rr]&&*+[F]{(0,2,0)}\ar[ru]& & &\\
& & & & (3,3,3)\ar[ruu] & (4,4,4)\ar[ru] & & (1,1,1)\ar[ru]& (2,2,2)\ar[ruu]& & & & &}
\end{displaymath}
\normalsize
\vspace{.1in}

\emph{We denote by  $(i,i,i)$ two different $2-$periodic complexes. One is the complex  $P_i \stackrel{0}{\rightarrow} P_i \stackrel{id}{\rightarrow} P_i$ and the other is the complex  $P_i \stackrel{id}{\rightarrow} P_i \stackrel{0}{\rightarrow} P_i$ for  $i=1, 2, 3, 4$. Precisely, the first $2$-periodic complex  that appear in the quiver contains the projective $P_2$ and is the complex $P_2 \stackrel{id}{\rightarrow} P_2 \stackrel{0}{\rightarrow} P_2$.}
\end{ej}

\section{On sectional paths in $\mathbf{C}_{\equiv m}(\mbox{proj} \, A)$}

Throughout this section,  we consider  $A$  a finite dimensional $k$-algebra over a field $k$.

The composition of irreducible morphisms on a sectional path in  ${\rm mod} \, A$ is non-zero, for any artin algebra $A$. This result was first proved by R. Bautista and S. Smal{\o} in 1983, see \cite{BSm}. Moreover, in \cite{IT} H. Igusa and G. Todorov also studied sectional paths  in ${\rm mod}\, A$ and  proved that the composition of $m$ irreducible morphisms on a  sectional path  does not belong  to the power $m+1$ of the radical of the category $\mbox{mod} \, A$ and as a consequence does not vanish.

On the other hand, in \cite{CPS1}, the authors analyzed this problem for a category of complexes of fixed size. They  proved  that the composition of irreducible morphisms on a sectional path in  $\mathbf{C}_{[1, n]}({\rm proj}\,A)$ can be zero and they studied conditions under which such compositions do not vanish. Precisely, the authors proved that if  $H$  is a finite dimensional hereditary algebra and $X_0 \rightarrow X_1  \rightarrow  \cdots \rightarrow X_m$ is a sectional path in $\mathbf{C}_{[1, n]}(\mbox{proj} \, H)$ then the composition of the irreducible morphisms in a sectional path does not vanish.

In this section, we shall prove that we have a similar result than the one mentioned above for the composition of irreducible morphisms on a sectional path in $\mathbf{C}_{\equiv m}(\mbox{proj} \, A)$,  whenever $A$ is a finite dimensional  $k$-algebra.
\vspace{.1in}

We recall the notion of sectional path.

\begin{defi}
A path $X_0\longrightarrow X_1  \longrightarrow  \cdots \longrightarrow X_m$  of irreducible morphisms between
indecomposable complexes in  $\mathbf{C}_{\equiv m}(\emph{proj} \, A)$ is sectional if
$\tau_{ \mathbf{C}_{\equiv m}}^{-1}X_i\not \simeq X_{i+2}$ for every $i=0, \dots, m-2$.
\end{defi}

S. Liu studied the behaviour of sectional paths  in a left or  a right Auslander-Reiten category, $\mathcal{A}$. Precisely, the mentioned author  proved that any sectional path  $X_0\stackrel{f_1}\rightarrow X_1  \stackrel{f_2}\rightarrow  \cdots \stackrel{f_m}\rightarrow X_m$ in these categories verify  that $f_m \cdots f_1  \notin \Re^{m+1}_{\mathcal{A}}(X_0,X_m)$, see   \cite[Lemma 2.7]{L}.
\vspace{.1in}

In \cite{CPS1}, the authors proved  that  $\mathbf{C}_{[1, n]}({\rm proj}\, H)$ is a left  Auslander-Reiten category if and only if $H$ is hereditary. Therefore, as an immediate consequence they got that the composition of irreducible morphisms in a sectional path in $\mathbf{C}_{[1, n]}({\rm proj}\, H)$ is always non-zero.
\vspace{.1in}

Next, we shall see that $\mathbf{C}_{\equiv m}(\mbox{proj} \, H)$ is a left  Auslander-Reiten category, whenever $H$ is a finite dimensional hereditary algebra.
\vspace{.1in}

We recall the following definitions from  \cite{L}.

\begin{defi} An object $X$ in a Krull-Schmidt category $\mathcal{A}$
is called pseudo-projective if there exists a right minimal almost split (sink)
monomorphism $M \rightarrow X$, and dually, it is called  pseudo-injective
if there exists a left minimal almost split (source) epimorphism $X \rightarrow N$.
\end{defi}

\begin{defi} A Krull-Schmidt category $\mathcal{A}$ is called
a left Auslander Reiten category if each indecomposable
object in $\mathcal{A}$ is either pseudo-projective or the
end-term of an almost split sequence.  $\mathcal{A}$ is called
a right
Auslander-Reiten category  if each indecomposable object in
$\mathcal{A}$ is either pseudo-injective or the starting term of
an almost split sequence. $\mathcal{A}$ is called
 an Auslander-Reiten
category if it is a left and a right Auslander-Reiten category.
\end{defi}

\begin{prop}\label{hereditary1} If $H$ is a finite dimensional hereditary algebra over a field $k$, then  $\mathbf{C}_{\equiv m}(\emph{proj} \, H)$ is a left  Auslander-Reiten category.

Conversely, if $\mathbf{C}_{\equiv m}(\emph{proj} \, H)$ is a left  Auslander-Reiten category for $m>2$, then $H$ is a hereditary algebra.
\end{prop}

\begin{proof} Let $H$ be a hereditary algebra. Then  s.gl.$\mbox{dim} \,H < \infty$. Consider $X$ an indecomposable complex  in $\mathbf{C}_{\equiv m}(\mbox{proj} \, H)$.

If $X$  is not projective then by \cite[Proposition 2.6]{CD} $X$ is the end-term of an almost split sequence in $\mathbf{C}_{\equiv m}(\mbox{proj} \, H)$. Otherwise, if $X$ is projective since the compression functor $\mathcal{F}_m: \C^{b}(\mbox{proj} \, H) \rightarrow \mathbf{C}_{\equiv m}(\mbox{proj} \, H)$ is dense because of  Theorem  \ref{dense}, then there is a complex $Z \in \C^{b}(\mbox{proj} \, H)$ such that $\mathcal{F}_m(Z)=X$. Furthermore, $Z$ is projective, because by \cite{PX} the functor $\mathcal{F}_m$ sends projective complexes in $\C^{b}(\mbox{proj} \, H)$ into projective complexes in $\mathbf{C}_{\equiv m}(\mbox{proj} \, H)$. Since $\rho: \mbox{rad}(Z)\rightarrow Z$ is a left almost split (sink) morphism ending in $Z$ then $\mathcal{F}_m(\rho):\mathcal{F}_m(\mbox{rad}(Z))\rightarrow \mathcal{F}_m(Z)$  is a left almost split (sink) morphism ending in $X$. Since $H$ is hereditary then the above sink morphisms are monomorphisms, because for every $P\in {\rm proj}\, H$ we have that the projective dimension of  $P/{\rm rad}\,P$ is less than or equal to one.

Conversely. Assume that $\mathbf{C}_{\equiv m}(\mbox{proj} \, H)$  is a left  Auslander-Reiten category  and $H$ is not hereditary. Then, there  exists an indecomposable projective $H$-module $P$  such that the projective dimension of $P/ {\rm rad} P$ is greater than one. Then we may assume that $$ R^{-s} \rightarrow  \dots \rightarrow R^{-2} \rightarrow R^{-1} \rightarrow P/ {\rm rad} P$$
\noindent is a minimal projective resolution for  $P/ {\rm rad} P$, with $R^{-1} \neq 0$ and $R^{-2}\neq 0$.

Consider the morphism $$\rho': (R^{-s} \rightarrow \dots \rightarrow R^{-2} \rightarrow R^{-1} \rightarrow P )\longrightarrow (0\rightarrow \dots \rightarrow 0 \rightarrow P \rightarrow P)$$
\noindent in $\C^b(\mbox{proj} \, H)$. Then we have that $\mathcal{F}_m(\rho')=\rho:\mbox{rad}(Z)\rightarrow Z $ if $Z= 0\rightarrow \dots \rightarrow 0 \rightarrow P \rightarrow P$.

 On the other hand, there exists a non-zero morphism $\alpha: (0\rightarrow \dots \rightarrow 0 \rightarrow R^{-2} \rightarrow 0 ) \longrightarrow (R^{-s}\rightarrow \dots \rightarrow R^{-2} \rightarrow R^{-1} \rightarrow P)$ in $\C^b(\mbox{proj} \, H)$ such that $\rho'\alpha =0$.
We illustrate the situation with the following diagram:

\begin{displaymath}
\def\objectstyle{\scriptstyle}
\def\labelstyle{\scriptstyle}
 \xymatrix {0 \ar[r] \ar[d] & \cdots \ar[r] & 0 \ar[r] \ar[d]    &    R^{-2}\ar[r]\ar[d]&  0 \ar[d] \\
 R^{-s}\ar[r]\ar[d]        & \cdots \ar[r] & R^{-2} \ar[r]\ar[d]& R^{-1}\ar[r]\ar[d]   & P \ar[d]\\
 0 \ar[r]                     &\cdots \ar[r]  & 0 \ar[r]           & P \ar[r]             & P  }
\end{displaymath}

\noindent  Then  $\mathcal{F}_m(\rho')\mathcal{F}_m(\alpha)=0$ with $\mathcal{F}_m(\alpha)$ non-zero. Therefore $\rho=\mathcal{F}_m(\rho')$ is not a sink monomorphism.

Moreover, $Z$ is not the ending of an almost split sequence, a contradiction to the assumption  that $\mathbf{C}_{\equiv m}(\mbox{proj} \, H)$ is a left Auslander-Reiten category. Therefore $H$ is hereditary.
\end{proof}

\begin{coro}\label{hereditary2} Let $H$  be a finite dimensional hereditary algebra and $X_0 \rightarrow X_1  \rightarrow  \cdots \rightarrow X_r$ be a sectional path in $\mathbf{C}_{\equiv m}(\emph{proj} \, H)$. If $f_i:X_{i-1} \rightarrow X_i$  are irreducible morphisms for $i=1, \cdots, r$
then $f_r \cdots f_1  \notin \Re^{r+1}_{\mathbf{C}_{\equiv m}(\emph{proj} \, H)}(X_0,X_r)$. In particular, the composition $f_r \cdots f_1 $ does not vanish.
\end{coro}

Now, we show an example where we can see that there are sectional paths in $\mathbf{C}_{\equiv m}(\mbox{proj} \, A)$ where the composition of the irreducible morphisms in it vanishes.

\begin{ej}
\emph{ Consider the Example \ref{ej1}. The path $$ (3,0,0,3 ,3) \stackrel{f}\rightarrow  (2,1,0,3,2) \stackrel{g}\rightarrow (1,1,0,0,1)$$ is clearly a sectional path since the end-points are projective-injective complexes and the  composition of the irreducible morphisms $f$ and $g$ vanishes.}
\end{ej}

The next lemma is useful to determined the shape of non-zero sectional paths in $\mathbf{C}_{\equiv m}(\mbox{proj} \, A)$, where s.gl.$\mbox{dim} \,A < \infty$.

\begin{lem} \label{J} Let $A$  be a finite dimensional $k$-algebra, with  s.gl.$\emph{dim} \,A < \infty$. If $X_0 \rightarrow X_1 \rightarrow \dots \rightarrow X_r$ is a sectional path in $\mathbf{C}_{\equiv m}(\emph{proj} \, A)$ with $X_i$ an indecomposable projective-injective complex for some $i \in \{0, \cdots, r\}$  then $i=0$ or $i=r$.
\end{lem}
\begin{proof} Consider $X_0 \rightarrow X_1 \rightarrow \dots \rightarrow X_r$ a sectional path in $\mathbf{C}_{\equiv m}(\mbox{proj} \, A)$ with $X_i$ an indecomposable projective-injective complex for some $i \in \{0, \cdots, r\}$. Let $X_j=J_{i_{0}}(P)$ be the projective-injective complex, where $P$ is an indecomposable projective $A$-module. Since  s.gl.$\mbox{dim} \,A < \infty$ then  the compression functor ${\mathcal F}_m: \C^{b}(\mbox{proj} \,A) \rightarrow \mathbf{C}_{\equiv m}(\mbox{proj} \,A)$ is dense. Moreover, ${\mathcal F}_m$ preserves indecomposable complexes,  projective  complexes and irreducible morphisms. Then we have that $Z_0 \rightarrow Z_1 \rightarrow \dots \rightarrow Z_r$ is a path in $\C^{b}(\mbox{proj} \,A)$, where ${\mathcal F}_m(Z_i)=X_i$ for $i=0, \dots, r$. Moreover, $Z_j$ is a projective-injective complex in $\C^{b}(\mbox{proj} \, A)$. By Theorem \ref{nu} the above path is sectional in $\C^{b}(\mbox{proj} \,A)$.

Now, by \cite[Proposition 8.7]{BSZ}, we know that  $\mbox{rad}(Z_j)=R_{i_{0}}(P)$ is an indecomposable complex in $\C_{[1,n]} (\mbox{proj} \,A)$, for $n$ arbitrarily large.  That is, we consider  $n$ such that the paths of irreducible morphisms $\mbox{rad}(Z_j) \rightarrow Z_j \rightarrow \tau_{\C^{b}(\mbox{proj} \,A)}^{-1} Z_{j-1}$ belongs to  $\C_{[1,n]}(\mbox{proj} \, A)$,  where $Z_{j}$ is the projective-injective complex and $Z_{j-1}=\mbox{rad}(Z_j)$.   Therefore, $\mbox{rad}(Z_j)$ is also indecomposable in $\C^{b}((\mbox{proj} \,A))$  and by Proposition \ref{ind},  ${\mathcal F}_m (\mbox{rad}(Z_j))$ is an indecomposable complex in $\mathbf{C}_{\equiv m}(\mbox{proj} \,A)$.

On the other hand, $\mbox{rad}(Z_j) \rightarrow  Z_j$  is irreducible in $\C^{b} (\mbox{proj} \,A)$.
Therefore, applying the compression functor we have that ${\mathcal F}_m (\mbox{rad}(Z_j)) \rightarrow  {\mathcal F}_m (Z_j)$ is irreducible in $\mathbf{C}_{\equiv m}(\mbox{proj} \, A)$.

Moreover, by \cite[Proposition 4.7]{CPS1}, the irreducible morphism  $\mbox{rad}(Z_j) \rightarrow  Z_j$ in $\C_{[1,n]} (\mbox{proj} \,A)$ has dimension one.
Dually, we can see that $Z_j \rightarrow  \tau_{\C^{b}}^{-1} Z_{j-1}$ in $\C_{[1,n]} (\mbox{proj} \,A)$ where $ Z_{j-1}=\mbox{rad}(Z_j)$, has dimension one. Hence,  the morphism $\mbox{rad}(Z_j) \rightarrow  Z_j$  has the same property in $\C^{b}(\mbox{proj} \,A)$. Applying the compression functor we have that ${\mathcal F}_m (\mbox{rad}(Z_j)) \rightarrow  {\mathcal F}_m (Z_j)$ has dimension one in $\mathbf{C}_{\equiv m}(\mbox{proj} \, A)$.  Dually, we can see that $ {\mathcal F}_m (Z_j) \rightarrow   {\mathcal F}_m ( \tau_{\C^{b}}^{-1} Z_{j-1})$ in $\C_{\equiv m} (\mbox{proj} \,A)$ has dimension one.

Now, assume that $j\neq r$.  If  $X_{r-1}$ is  projective-injective  then the path $X_{r-2} \rightarrow X_{r-1} \rightarrow X_r$ is not sectional, since $X_{r-2}=\mbox{rad}(X_{r-1})$, and $X_r = \tau_{\C{\equiv m}}^{-1}X_{r-2}$, an absurdly. Thus, $i \neq  r-1$.

A similar proof can be done for any $i \in \{1, \dots, r-2\}$. Since by hypothesis there is an $X_i$  projective-injective then we have that $i=0$, proving the result.
\end{proof}

\begin{prop}\label{nonu}
Let $A$ be a finite dimensional  $k$-algebra with s.gl.$\emph{dim} \,A < \infty$. Let $X_0 \rightarrow X_1 \rightarrow \dots \rightarrow X_r$  be a non-zero sectional path in $\mathbf{C}_{\equiv m}(\emph{proj} \, A)$ with $m > 1$. Then, one of the following statements hold.
\begin{itemize}
\item [(a)] The complexes $X_1, X_2, \cdots, X_{r}$ are not projective-injective.
\item [(b)] The complexes $X_0, X_1, \cdots, X_{r-1}$ are not projective-injective.
\item [(c)] $X_0$ and $X_r$  are projective-injective complexes and $X_1, \cdots, X_{r-1}$ are not projective-injective.
\end{itemize}
\end{prop}

\begin{proof} Let $X_0 \rightarrow X_1 \rightarrow \dots \rightarrow X_r$ be a non-zero sectional path in $\mathbf{C}_{\equiv m}(\mbox{proj} \, A)$ with $m > 1$. If there are not projective-injective complexes in such a path then condition (a)  and (b) hold.

Assume that there is an  $i \in \{0, \dots, r\}$ such that $X_i$ is  projective-injective. Furthermore, let   $i \in \{0, \dots, m\}$
be the least integer such that $X_i$ is projective-injective. By Lemma \ref{J} we have that $i=0$ or $i=r$.

If $i =0$ then $X_0$ is projective-injective and for $j>0$ we have the following two situations because of Lemma \ref{J}: the path is of the form $X_0 \rightarrow \dots \rightarrow X_r$, where $X_j$ are not projective-injective for $j=1, \dots, r-1$ and $X_r$ is projective-injective; or
of the form  $X_0 \rightarrow \dots \rightarrow X_r$ with $X_1, \dots, X_r$ not projective-injective. Then conditions (c) and (a) hold, respectively. In case that $i=r$ then $X_0, \cdots, X_{r-1}$ are not projective-injective and (b) holds.
\end{proof}


\begin{thebibliography}{99}


 \bibitem{A} H. Asashiba. \emph{A generalization of Gabriel's Galois covering functors and derived equivalences},
Journal of Algebra 334,   (2011), 109-149.

\bibitem{AR} M. Auslander, I. Reiten. \emph{Stable equivalence of dualizing $R$-varieties}. Advances in Math. 12, (1974), 306-366.

 \bibitem{BL} R. Bautista, S. Liu. \emph{Covering theory for linear categories with application to derived categories},
Journal of Algebra 406,   (2014), 173-225.

\bibitem{BSm} R. Bautista, S. O. Smal\o. \emph{Non-existent cycles}. Communications in
Algebra 11,  (1983), 1755-1767.

\bibitem{BSZ} R. Bautista, M. J. Souto Salorio, R. Zuazua. \emph{Almost split sequences for complexes of fixed size}. Journal of
Algebra {287}, (2005), 140-168.


\bibitem{B} T. Bridgeland. \emph{Quantum groups via Hall algebras of complexes}. Ann. of Math. (2) 177, (2013), 1-21.

\bibitem{CGP2} C. Chaio, A. Gonz\'alez Chaio,  I. Pratti. \emph{On non-homogeneous tubes and components of type $\mathbb{Z}A_{\infty}$  in the bounded  derived category}.  Algebras and Representation Theory 24, (2), (2021), 327-356.

\bibitem{CGP3} C. Chaio, A. Gonz\'alez Chaio,  I. Pratti.     \emph{Sections in the bounded derived category
of piecewise hereditary algebras}. Journal of Algebra and Its Applications (2024).

\bibitem{CPS1} C. Chaio, I. Pratti, M. J. Souto Salorio. \emph{On sectional paths in a category of complexes of fixed size}. Algebras and Representation Theory 20,  (2017), 289-311.

\bibitem{Ch} Q. Chen. \emph{Almost split sequences of $m-$cyclic complexes}  Arch. Math. 119, (2022), 569-581

\bibitem{CD}  Q. Chen,  B. Deng.  \emph{Cyclic complexes, Hall polynomials and simple Lie algebras}. Journal of Algebra 440, (2015), 1-32.

\bibitem{F} H. Fu.  \emph{On root categories of finite-dimensional algebras}. Journal of  Algebra 370, (2012), 233-265.

\bibitem{G} P. Gabriel, A. V. Roiter. \emph{Representations of finite dimensional algebras}. Encyclopaedia
of the Mathematical Sciences 73, A.I. Kostrikin and I. V. Shafarevich (Eds.), Algebra VIII,
Springer (1992).

\bibitem{GM} H. Girardo, H. Merklen. \emph{Irreducible morphism of the category of complexes}. Journal of Algebra 321,  (2009), 2716-2736.


\bibitem{HZ} Y. Han, N. Zhang. \emph{ Construction of dualizing categories by tensor products of categories}. Scientia Sinica Mathematica 48, (11), (2018), 1699-1716.

\bibitem{H} D. Happel. \emph{Triangulated categories in the representation theory of finite dimensional algebras}. London Math. Soc. Lecture Note Ser. 119, Cambridge, (1988).

\bibitem{IT} K. Igusa, G. Todorov. \emph{ A characterization of finite Auslander-Reiten quivers}. Journal Algebra {89}, (1984),  148-177.

\bibitem{Ke} B. Keller. \emph{ On triangulated orbit categories}. Doc. Math. 10, (2005), 551–581.

\bibitem{L} S. Liu. \emph{Auslander-Reiten theory in
a Krull-Schmidt category}. Sao Paulo J. Math. Sci. 4, (2010),  425-472.

\bibitem{PX} L. Peng, J. Xiao.   \emph{Root Categories and Simple Lie Algebras}. Journal of Algebra 198, (1997), 19–56.

\bibitem{RZ} C. M. Ringel, P.  Zhang. \emph{Representations of quivers over the algebra of dual numbers}. Jounal of Algebra 475, (2017), 327–360.

 \bibitem{R}  K. W. Roggenkamp. \emph{Auslander-Reiten triangles in derived categories}. Forum Math.  8 (5), (1996), 509-533.

\bibitem{Sa}  S. Saito.   \emph{Tilting objects in periodic triangulated categories}.  arXiv:2011.14096, (2021).

\bibitem{Sk} A. Skowro\'nski.\emph{ On algebras with finite strong global dimension}. Bull. Polish Acad. Sci. 35, (1987), 539–547.


\bibitem{S}  T. Stai.   \emph{Differential Modules over Quadratic Monomial Algebras}. Algebras and Representation Theory 20, (2017), 1239-1247.

\bibitem{S-18}  T. Stai.   \emph{The triangulated hull of periodic complexes.} Math. Res. Lett 25, (1), (2018), 199-236.


\bibitem{Z} X. Zhao. \emph{A note on the equivalence of m-periodic derived categories}. Science China Mathematics 57, (11), (2014), 2329-2334.


\end{thebibliography}
\end{document}